\documentclass[11pt, draft]{amsart}[2000]

\usepackage{amsmath}
\usepackage{amssymb}
\usepackage{amsthm}  
\usepackage[all]{xypic}
\usepackage{enumitem}
\usepackage{mathtools}
\usepackage{color}
\usepackage{bbm}
\usepackage{tikz-cd}
\usepackage{soul}
\usepackage[all]{xy}

\theoremstyle{plain}

\newtheorem{Thm}{Theorem}[section]
\newtheorem{Cor}[Thm]{Corollary}
\newtheorem{Lemma}[Thm]{Lemma}

\newtheorem{Prop}[Thm]{Proposition}

\theoremstyle{definition}
\newtheorem{Def}[Thm]{Definition}

\newtheorem{Exl}[Thm]{Example}
\newtheorem{Rmk}[Thm]{Remark}

\newtheorem{Hypo}[Thm]{Hypothesis}

\numberwithin{equation}{section}

\newcommand{\A}{{\mathcal{A}}}
\newcommand{\B}{{\mathcal{B}}}
\newcommand{\C}{{\mathcal{C}}}
\newcommand{\D}{{\mathcal{D}}}
\newcommand{\E}{{\mathcal{E}}}
\newcommand{\F}{{\mathcal{F}}}

\renewcommand{\H}{{\mathcal{H}}}
\newcommand{\I}{{\mathcal{I}}}
\newcommand{\J}{{\mathcal{J}}}
\newcommand{\K}{{\mathcal{K}}}
\renewcommand{\L}{{\mathcal{L}}}

\renewcommand{\P}{{\mathcal{P}}}

\newcommand{\T}{{\mathcal{T}}}

\newcommand{\bC}{\mathbb{C}}

\newcommand{\bN}{\mathbb{N}}
\newcommand{\bT}{\mathbb{T}}
\newcommand{\bZ}{\mathbb{Z}}

\newcommand{\Alg}{\operatorname{Alg}}

\newcommand{\Aut}{\operatorname{Aut}}
\newcommand{\alg}{\operatorname{alg}}

\newcommand{\id}{{\operatorname{id}}}

\newcommand{\Span}{\operatorname{span}}

\newcommand{\ca}{C^*}
\DeclareMathOperator{\Ad}{Ad}
\DeclareMathOperator{\reg}{reg}
\DeclareMathOperator{\fin}{fin}
\DeclareMathOperator{\sing}{sing}

\begin{document}
	
	\title[Classification of Pimsner algebras]{Classification of irreversible and reversible Pimsner operator algebras}
	
	\author[A. Dor-On]{Adam Dor-On}
	\address{Department of Mathematical Sciences\\ University of Copenhagen \\ Copenhagen \\ Denmark.}
	\email{adoron@math.ku.dk}
	
	\author[S. Eilers]{S{\o{}}ren Eilers}
	\address{Department of Mathematical Sciences\\ University of Copenhagen \\ Copenhagen \\ Denmark.}
	\email{eilers@math.ku.dk}
	
	\author[S. Geffen]{Shirly Geffen}
	\address{Department of Mathematics\\ Ben-Gurion University of the Negev \\ Be'er Sheva\\ Israel.}
	\email{shirlyg@post.bgu.ac.il}

	
	\thanks{The first author was supported by NSF grant DMS-1900916 and by the European Union's Horizon 2020 Marie Sklodowska-Curie grant No 839412. The second author was supported by the DFF-Research Project 2 `Automorphisms and Invariants of Operator Algebras', no. 7014-00145B, and by the DNRF through the Centre for Symmetry and Deformation (DNRF92). The third author was supported by a Negev fellowship, a Minerva fellowship programme, an ISF grant no. 476/16 and the DFG through SFB 878 and EXC 2044 Mathematics M\"{u}nster: Dynamics--Geometry--Structure.}

	\begin{abstract}
	Since their inception in the 30's by von Neumann, operator algebras have been used to shed light in many mathematical theories. Classification results for self-adjoint and non-self-adjoint operator algebras manifest this approach, but a clear connection between the two has been sought since their emergence in the late 60's.
	
	We connect these seemingly separate types of results by uncovering a hierarchy of classification for non-self-adjoint operator algebras and $C^*$-algebras with additional $C^*$-algebraic structure. Our approach naturally applies to algebras arising from $C^*$-correspondences to resolve self-adjoint and non-self-adjoint isomorphism problems in the literature. We apply our strategy to completely elucidate this newly found hierarchy for operator algebras arising from directed graphs.
	\end{abstract}
	
		
	\maketitle

	\renewcommand{\O}{{\mathcal{O}}}
	
	\section{Introduction}
	
	Originating in Elliott's work in the 70s, the endeavor to classify simple $C^*$-algebras via K-theory provided increasingly sophisticated classification and structural results, which have led to fruitful applications in dynamical systems and group theory. One such application is the classification of multivariable Cantor minimal systems \cite{GMPS10} by Giordano, Matui, Putnam and Skau. In a recent breakthrough due to Tikuisis, White and Winter \cite{TWW17}, Elliott's classification program is now nearly completed, and Rosenberg's conjecture on quasi-diagonality of $C^*$-algebras of discrete, amenable groups has been verified.
	
	In broad terms, $C^*$-rigidity is the concept that objects can be recovered (up to some equivalence), from associated $C^*$-algebraic data. This was established for $C^*$-algebras arising from various structures such as dynamical systems \cite{GPS95, Kri76, Ror95}, groupoids \cite{CRST+, Ren08}, number fields \cite{Li14, Li16} and more. 
	
	On the other hand, non-self-adjoint operator algebras are subalgebras of $C^*$-algebras and provide invariants for irreversible objects such as multivariable one-sided dynamical systems \cite{DK11a}, analytic varieties \cite{DRS11} and Markov chains \cite{DOM14}. The study of such operator algebras is motivated from single operator theory and complex analysis, partly in hope of resolving the unyielding Invariant Subspace Problem.
	
	Non-self-adjoint classification started with a paper of Arveson \cite{Arv67} on classification of operator algebras associated to measure preserving automorphisms. It was later realized that classification problems for many algebras can be put in the unified context of tensor algebras of $C^*$-correspondences \cite{MS00}, and in many concrete cases such problems were resolved \cite{DKat08, DK11a, DO18, KK04}.
	
	Operator algebras arising from graphs form one of the most important classes for classification of self-adjoint and non-self-adjoint operator algebras. In forthcoming work with Ruiz and Sims, the second named author has observed that graph $C^*$-algebras of amplified finite graphs (so that all edges have infinite multiplicity), as in \cite{ERS12}, together with their diagonal and gauge action can recover the amplified graph. On top of this, by using classification of KMS states on Toeplitz graph algebras of finite graphs, it was recently shown in \cite[Theorem 3 (1)]{BLRS19} that the vertex diagonal and gauge action completely recovers the graph.
	 
	Alternatively, recovering graphs from their non-self-adjoint graph tensor algebras is possible for arbitrary graphs by work of Solel \cite{Sol04} or for weaker notions of isomorphisms by work of Katsoulis and Kribs \cite{KK04}. From such concrete cases we see that invariants produced by classification of $C^*$-algebras with additional structure are drawing close to those produced by classification of non-self-adjoint operator algebras. Thus, a natural question is to ask for an exact connection between these seemingly separate rigidity type results. In this paper we show that it is more than just a coincidence that directed graphs can be recovered from both their irreversible algebra and their reversible algebra together with additional structure.
	
	A central study that connects $C^*$-algebras with non-self-adjoint operator algebras is Arveson's non-commutative boundary theory, which he developed and applied in several papers \cite{Arv69, Arv72, Arv98}. Classical boundaries of function algebras were the subject of intense research in the 50's and 60's, and are related to convexity and approximation theory via Choquet theory \cite{CM63}. The non-commutative generalization of the Shilov boundary is called the $C^*$-envelope, and was first shown to exist through Hamana's injective envelope \cite{Ham79}. The $C^*$-envelope is defined to be the smallest $C^*$-algebra containing the given operator algebra in a reasonable sense and provides a fruitful connection between $C^*$-algebras and non-self-adjoint operator algebras (see \cite{DK20, Kat17}). In this paper we apply ideas from Arveson's non-commutative boundary theory to uncover a precise hierarchy between classification of irreversible algebras and classification of reversible algebras with additional structure.
	
For what follows we will assume familiarity with the theory of Hilbert $C^*$-modules as presented in \cite{Lan95,MT05}.
	
\begin{Def}\label{d: correspondence} A \textit{$C^*$-correspondence over a $C^*$-algebra $\A$} is a right Hilbert $\A$-module $E$ together with a *-representation $\phi_E: \A\to \L(E)$, where $\L(E)$ denotes the $C^*$-algebra of adjointable operators on $E$.
\end{Def}

 A $C^*$-correspondence $E$ over $\A$ comes equipped with the operator space structure it inherits as a subspace of the linking $C^*$-algebra 
$$
\D_E=\begin{bmatrix}
\A & E^*\\
E & \K(E)
\end{bmatrix} \subseteq \L(\A\oplus E),
$$
for a more detailed discussion on linking algebras see \cite{BGR77}. When the context is clear, we write $a \xi$ to mean $\phi_E(a) \xi$ for $a\in \A$ and $\xi \in E$. To define our algebras universally, we will need the following notions of representations of $C^*$-correspondences from \cite{MS98}. We note immediately that what we call a \emph{rigged} representation here is often referred to as an \emph{isometric} representation in the literature. The reason for choosing this term is that a representation of $E$ can be isometric (or completely isometric) as a map on $E$ with its given operator space structure without being rigged.
	
	\begin{Def}
		Let $E$ be a $C^*$-correspondence over $\A$, and $\B$ some $C^*$-algebra. A (completely contractive) representation of $E$ is a pair $(\pi, t)$ such that $\pi : \A \rightarrow \B$ is a *-homomorphism and $t : E \rightarrow \B$ is a completely contractive linear map such that
		\begin{enumerate}
			\item
			$\pi(a)t(\xi) \pi(b) = t(a \cdot \xi \cdot b)$ for $a,b\in \A$ and $\xi \in E$.
		\end{enumerate}
		We say that $(\pi,t)$ is a \emph{rigged representation} if additionally
		\begin{enumerate}
			\item[(2)] $t(\xi)^*t(\eta) = \pi(\langle \xi,\eta \rangle)$
		\end{enumerate}
		We say that $(\pi, t)$ is \textit{injective} if $\pi$ is an injective *-homomorphism.
		We denote by $C^*(\pi,t)$ and $\overline{\Alg}(\pi,t)$ the $C^*$-algebra and the norm-closed operator algebra, respectively, generated by the images of $\pi$ and $t$ inside $\B$.
	\end{Def}
	
	The Toeplitz algebra $\T(E)$ is then the universal $C^*$-algebra generated by \emph{rigged} representations of $E$, and the tensor algebra $\T_+(E)$ is the universal operator algebra generated by all representations of $E$. For each $n\in \bN$, denote by $t^{(n)}: E^{\otimes n}\to \B$ the map uniquely determined on simple tensors by $t^{(n)}(\xi_1 \otimes \ldots \otimes \xi_n) = t(\xi_1) \circ \ldots \circ t(\xi_n)$ (See Subsection \ref{ss: Toeplitz and Tensor algebras of C*-correspondences} for more details). Suppose now that $(\pi,t)$ is a rigged representation such that $\T(E) \cong C^*(\pi,t)$. In this case we say that $(\pi,t)$ is universal. It then follows from the definition of rigged representation that 
	$$
	\mathcal{T}(E)=\overline{\Span}\{t^{(n)}(\xi) t^{(m)}(\eta)^*: \xi\in E^{\otimes n}, \eta\in E^{\otimes m}, m,n\geq 0 \}.
	$$
	Moreover, universality of $\T(E)$ implies that it comes equipped with a point-norm continuous circle action $\gamma: \bT \to \Aut(\T(E))$ given by
	$$
	\gamma_z(\pi(a))= \pi(a), \ \text{ and } \ \  \gamma_z(t(\xi))=z \cdot t(\xi), \ \ \text{for} \ z\in \bT, \xi\in E, a\in \A.
	$$
	
	\begin{Def} Let $E$ be a $C^*$-correspondence over $\A$. Denote by $\gamma$ the circle action on $\T(E)$. For each $n\in \bZ$, the $n$-spectral subspace for $\gamma$ is defined by 
		$$
		\T(E)_n=\{ T \in \T(E): \gamma_z(T)=z^n \cdot T \text{ for } z\in \bT \}.
		$$
	\end{Def}
	
	The circle action $\gamma$ provides $\T(E)$ with Fourier coefficients $\Phi_n$ given by 
	$$
	\Phi_n(T)=\int_{\bT} \gamma_z(T)z^{-n}dz,
	$$
	where $dz$ denotes normalized Haar measure on $\bT$. Using these Fourier coefficients, for each $T\in \T(E)$ the Ces\`aro sums $\sum_{k=-n}^n \big(1 - \frac{|k|}{n})\Phi_n(T)$ converge \emph{in norm} to $T$. In particular, we see that $T=0$ if and only if $\Phi_n(T) = 0$ for all $n\in \bZ$. 
	Notice further that $\Phi_n$ is an idempotent, and for a universal rigged representation $(\pi,t)$ the image of $\Phi_n$ contains $\overline{\Span}\{ t^{(k)}(\xi)t^{(l)}(\eta)^* : \xi \in E^{\otimes k}, \eta \in E^{\otimes l}, k-l = n \}$. Hence, by Ces\`aro approximation we get that
	$$
	\T(E)_n = \overline{\operatorname{\Span}}\{ t^{(k)}(\xi) t^{(l)}(\eta)^* : \xi \in E^{\otimes k}, \eta \in E^{\otimes l}, k-l = n \}.
	$$
	It then follows that $\{ \T(E)_n \}_{n\in \bZ}$ is a \emph{topological} grading for $\T(E)$ in the sense of Exel \cite[Definition 19.2]{Exelbook}, where $\Phi_0$ is the conditional expectation which indicates that the grading is topological.
	
	Since the tensor algebra $\T_+(E)$ is naturally a subalgebra of $\T(E)$ (see Subsection \ref{ss: Toeplitz and Tensor algebras of C*-correspondences}), it is invariant under the circle action $\gamma$ as a subalgebra of $\T(E)$, and there are spectral subspaces defined for $\T_+(E)$ as well. For each $n\in\ \bN $, the $n$-spectral subspace for $\T_+(E)$ is given by
	$$
	\T_+(E)_n=\{ T \in \T_+(E): \gamma_z(T)=z^n \cdot T \text{ for } z\in \bT \}.
	$$
	We say that $\{ \T_+(E)_n \}_{n\geq 0}$ is the \textit{grading} for $\T_+(E)$, and by \cite[Proposition 4.2]{DO18} we have concretely, when $(\pi,t)$ is universal, that 
$$
\T_+(E)_n= \{ t^{(n)}(\xi): \xi\in E^{\otimes n} \}.
$$

	Toeplitz-Pimsner algebras have a canonical quotient, also known as the Cuntz-Pimsner algebra originally defined by Pimsner in \cite{Pim95} and refined by Katsura in \cite{Kat04b}. These algebras generalize many constructions of operator algebras in the literature.
	
	\begin{Def}\label{d:Katsura-ideal}
		For a $C^*$-correspondence $E$ over $\A$, we define \emph{Katsura's ideal} $J_E$ in $\A$ by
		$$J_E\coloneqq  \{a\in \A: \phi_E(a)\in \K(E) \text{ and } ab=0 \text{ for all } b\in \ker \phi_E \}.$$
	\end{Def}
	
	For a rigged representation, $(\pi,t)$ of a $C^*$-correspondence $E$ over $\A$, it is a standard fact that there is a well-defined *-homomorphism $\psi_t: \K(E)\to \B$ given by $\psi_t(\theta_{\xi,\eta})=t(\xi)t(\eta)^*$ for $\xi,\eta\in E$.
	
	\begin{Def}\label{d:cov-rep}
		A rigged representation $(\pi,t)$ is said to be \textit{covariant} if $\pi(a)=\psi_t(\phi_E(a))$, for all $a\in J_E$.
	\end{Def}
	
	The Cuntz-Pimsner algebra $\O(E)$ is defined as the universal $C^*$-algebra generated by \textit{covariant} representations of $E$. Thus, $\O(E)$ is the quotient of $\T(E)$ by the ideal of relations $\J_E$ generated by $\{ \psi_T(\phi_E(a))-a: a\in J_E \}$. The following notions of isomorphisms between tensor and Toeplitz algebras are essential to our paper.

\begin{Def}\label{d:GradIso}
		Let $E, F$ be $C^*$-correspondences over $\A$ and $\B$, respectively. 
		\begin{enumerate}
			\item 
			A \textit{base-preserving isomorphism} between $\T(E)$ and $\T(F)$ is an isomorphism:
			$$
			\varphi: \T(E)\to \T(F)  \text{ s.t. } \varphi(\A)=\B.
			$$
			\item
			A \textit{base-preserving isomorphism} between $\T_+(E)$ and $\T_+(F)$ is an isomorphism: 
			$$
			\varphi: \T_+(E)\to \T_+(F)  \text{ s.t. } \varphi(\A)=\B.
			$$
			\item
			A \textit{graded isomorphism} between $\T(E)$ and $\T(F)$ is an isomorphism:
			$$
			\varphi: \T(E)\to \T(F)  \text{ s.t. } \varphi(\T(E)_n)=\T(F)_n \text{ for all } n\in \mathbb{Z}.
			$$
			\item
			A \textit{graded isomorphism} between $\T_+(E)$ and $\T_+(F)$ is an isomorphism: 
			$$
			\varphi: \T_+(E)\to \T_+(F)  \text{ s.t. } \varphi(\T_+(E)_n)=\T_+(F)_n \text{ for all } n\in \mathbb{N}.
			$$
		\end{enumerate}
	\end{Def}
	
By \cite[Theorem 3]{Rae18} we see that there is a bijective correspondence between topologically $\mathbb{Z}$-graded $C^*$-algebras, with graded *-homomorphisms and $C^*$-algebras equipped with a circle action, together with equivariant *-homomorphisms. In particular, an isomorphism $\varphi:\T(E)\to \T(F)$ is graded, if and only if it is equivariant in the sense that $\varphi\circ \gamma^E_z=\gamma^F_z\circ \varphi$, for all $z\in \bT$.

	\begin{Def} \label{D:ses-iso}
		Let $E$ and $F$ be $C^*$-correspondences over $\A$ and $\B$ respectively. We say that $\T(E)$ and $\T(F)$ are \emph{short-exact sequence isomorphic} (s.e.s. isomorphic) if there is a *-isomorphism $\varphi : \T(E) \rightarrow \T(F)$ such that $\varphi$ restricts to a *-isomorphism from $\J_E$ onto $\J_F$. In this case we call $\varphi$ an s.e.s. isomorphism.
	\end{Def}
	
	Clearly the existence of an s.e.s. isomorphism is equivalent to having an isomorphism of the short exact sequences,
	\begin{equation} \label{eq:ses-iso}
	\begin{tikzcd}
	0 \arrow[r] & \J_E \arrow[r] \arrow[d, "\varphi_0"] & \T(E) \arrow[r] \arrow[d, "\varphi"] & \O(E) \arrow[d, "\overline{\varphi}"] \arrow[r] & 0 \\
	0 \arrow[r] & \J_F \arrow[r] & \T(F) \arrow[r] & \O(F) \arrow[r] & 0,
	\end{tikzcd}
	\end{equation}
	where $\varphi_0$ is the restriction of $\varphi$ to $\J_E$, and $\overline{\varphi}$ is the induced map on the Cuntz-Pimsner algebras. Given $C^*$-correspondences $E$ and $F$ over $C^*$ algebras $\A$ and $\B$, consider the following notions of isomorphisms
	\newpage
\begin{enumerate}
			\item
			$E$ and $F$ are unitarily isomorphic $C^*$-correspondences.
			
			\item
			$\T_+(E)$ and $\T_+(F)$ are graded completely isometrically isomorphic.
			\item
			$\T_+(E)$ and $\T_+(F)$ are completely isometrically isomorphic.
			\item
			$\T(E)$ and $\T(F)$ are base-preserving graded isomorphic.
			\item
			$\T(E)$ and $\T(F)$ are base-preserving s.e.s. isomorphic.
\end{enumerate}
The first achievement in our paper is Corollary \ref{C:main-square} which establishes the following hierarchy: 
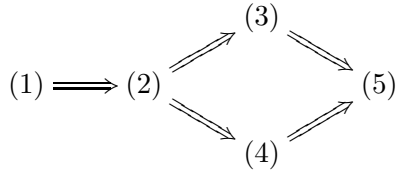
\[
\xymatrix@R=0.3cm{
&&(3)\ar@{=>}[dr]\\
(1)\ar@{=>}[r]&(2)\ar@{=>}[ur]\ar@{=>}[dr]&&(5)\\
&&(4)\ar@{=>}[ur]}
\]
This allows us to investigate completely isometric isomorphism problems for tensor algebras via structure-preserving isomorphisms on Toeplitz-Pimsner algebras.

In Section \ref{S: Graded isomorphisms} we study the graded and base-preserving picture, and when our $C^*$-correspondences are over compact operator subalgebras we show in Corollary \ref{C:equiv-iso} that $(4)$ implies $(1)$.

K-theory techniques from \cite{Kat04b} are employed in Section \ref{S: K-theory} to determine how isomorphism of short exact sequences as in diagram \eqref{eq:ses-iso} promotes to isomorphisms of associated K-groups. More specifically, in Theorem \ref{T:K-theory-six-term} we show that stable base-preserving s.e.s. isomorphisms induce isomorphisms of associated six-term exact sequences in K-theory that only involve the coefficient algebras, Katsura ideals and the Cuntz-Pimsner algebras. In Proposition \ref{P:compute-tau}, when $E$ and $F$ are $C^*$-correspondences over compact operator subalgebras, we are able to compute a natural connecting map between $K_0$ of Katsura ideals that is useful for later computations.

We apply our techniques in the context of graph tensor and Toeplitz-Cuntz-Krieger algebras in Section \ref{S: hierarchy for graph algebras}. We first use Corollary \ref{C:equiv-iso} to get a quick extension of \cite[Theorem 3 (1)]{BLRS19} to arbitrary graphs. More specifically, for any two graphs $G$ and $G'$, and their associated $C^*$-correspondences $E=X(G)$ and $F=X(G')$ we show that items $(1),(2)$ and $(4)$ are all equivalent to $G$ and $G'$ being isomorphic. Thus, Corollary \ref{C:equiv-iso} can be viewed as a further generalization of \cite[Theorem 3 (1)]{BLRS19} to $C^*$-correspondences over compact operator subalgebras.

For operator algebras associated to row-finite directed graphs we completely resolve the hierarchy of isomorphism problems. We show that every base-preserving isomorphism of Toeplitz-Cuntz-Krieger algebras of row-finite graphs is automatically an s.e.s. isomorphism, even after stabilization. Hence, item $(5)$ and its stable analogue in our hierarchy have simpler descriptions. Our main result for graph algebras of row-finite graphs is Theorem \ref{T: stable-graphs-base-pres} where we show for row-finite graphs $G$ and $G'$ that not only all of items $(1)-(5)$, but also that stable isomorphisms versions of items $(3)$ and $(5)$ are all equivalent to $G$ and $G'$ being isomorphic directed graphs. The stable completely isometric isomorphism problem for tensor algebras was impervious to standard methods from classification of non-self-adjoint operator algebras because these often relied on finite dimensional representation techniques. It is clear that all (completely contractive) representations of stabilized tensor graph algebras are on infinite dimensions so that finite-dimensional representation techniques normally cannot be applied. 

Together with Corollary \ref{C:main-square}, the above clearly shows that the classification of non-self-adjoint algebras goes hand in hand with classification of $C^*$-algebras with additional structure, and that previously intractable problems from non-self-adjoint classification can now be resolved by using classification techniques from $C^*$-algebra theory.

Finally, in Example \ref{ex:non-iso-base-iso} we show the limitations of our techniques in resolving general tensor algebra classification problems. More precisely, there exist two non-isomorphic amplified graphs $G$ and $G'$ (amplified in the sense that all edges have infinite multiplicity) such that their Toeplitz-Cuntz-Krieger algebras are base-preserving s.e.s. isomorphic. Together with \cite[Theorem 2.11]{KK04}, this shows that items $(5)$ and $(3)$ in our hierarchy above are generally not equivalent without some regularity assumption on the $C^*$-correspondences.

This paper contains six sections, including this introduction section. In Section \ref{S:prelim} we give some necessary material on dilation extreme representations, as well as Toeplitz-Pimsner, Cuntz-Pimsner and tensor algebras. In Section \ref{S: Hierarchy of isomorphism problems} we establish the main hierarchy of isomorphism problems. In Section \ref{S: Graded isomorphisms} we focus on graded isomorphisms and deduce rigidity results when the $C^*$-correspondences are over subalgebras of compact operators. In Section \ref{S: K-theory} we make essential connections between K-theory isomorphisms and stable base-preserving s.e.s. isomorphisms. Finally in Section \ref{S: hierarchy for graph algebras} we apply our techniques to resolve graded isomorphism problems for tensor and Toeplitz-Cuntz-Krieger algebras, and conclude with the resolution of the hierarchy for operator algebras associated to row-finite graphs.

	
	\section{Preliminaries} \label{S:prelim}
	
	
	\subsection{Dilation extremity and $C^*$-envelope}
	
	We explain how to define the notions of dilation extremity (normally called maximality in the literature) and the unique extension property for representations of \emph{not-necessarily-unital} operator algebras, in a way that yields the same theory as in the unital case. We refer the reader to \cite[Subsection 2.2]{DS18} for more details. For an operator algebra $\A$, by a \emph{representation} of $\A$ we shall henceforth mean a completely contractive homomorphism $\rho : \A \rightarrow B(\H)$. When $\rho: \A \rightarrow B(\H)$ is a representation, a \textit{dilation} $\pi : \A \rightarrow B(\K)$ is a representation such that $\H \subseteq \K$ and $\rho(a) = P_{\H} \pi(a)|_{\H}$.
	
	If $\A$ is an operator algebra and $\iota :\A \rightarrow \B$ is a completely isometric homomorphism such that $\B = C^*(\A)$, we say that the pair $(\iota,\B)$ is a $C^*$-cover. The $C^*$-envelope is defined as the smallest $C^*$-cover $(\kappa, C^*_e(\A))$ among all $C^*$-covers in the sense that for any $C^*$-cover $(\iota,\B)$ we have a natural quotient map $q_e : \B \rightarrow C^*_e(\A)$ such that $\kappa = q_e \circ \iota$. The $C^*$-envelope for unital operator algebras was first proven to exist by Hamana \cite{Ham79}, and a dilation theoretic proof was later given by Dritchel and McCullough \cite{DM05}. When discussing the $C^*$-envelope and other $C^*$-covers we will often suppress the maps $\iota$ and $\kappa$ and think of them as inclusions.
	
	If $\A \subseteq B(\H)$ is a \emph{non-unital} operator algebra generating a $C^*$-algebra $\B$ then by Meyer's theorem \cite[Section 3]{Mey01} every representation $\varphi: \A \rightarrow B(\K)$ extends to a unital representation $\varphi^1$ on the \emph{unitization} $\A^1=\A \oplus \mathbb{C}I_{\H}$ of $\A$ by setting $\varphi^1(a+\lambda I_{\H}) = \varphi(a) + \lambda I_{\K}$. This shows that there is a unique operator algebra structure on the unitization of $\A$, and yields the following Arveson extension theorem for not-necessarily unital operator algebras.
	
	\begin{Cor} \label{c:non-unital-arv}
		Let $\A \subseteq B(\H)$ be an operator algebra generating a $C^*$-algebra $\B$, and let $\varphi : \A \rightarrow B(\K)$ be a representation of $\A$. Then there is a completely positive contractive map $\widetilde{\varphi} : \B \rightarrow B(\K)$ such that $\widetilde{\varphi}|_{\A} = \varphi$.
	\end{Cor}
	
	We then define the unique extension property and dilation extremity (maximality) for not-necessarily-unital operator algebras, in a way that extends the same definitions for unital operator algebras and requiring that the maps are also unital.
	
	\begin{Def} \label{def:non-unital-UEP}
		Let $\A \subseteq B(\H)$ be an operator algebra generating a $C^*$-algebra $\B$. Let $\rho : \A \rightarrow B(\K)$ be a representation.
		\begin{enumerate}
			\item
			We say that $\rho$ has the \emph{unique extension property} (UEP) if every completely positive contractive map $\pi: \B \rightarrow B(\K)$ extending $\rho$ is a *-representation.
			\item
			We say that $\rho$ is \emph{dilation extreme} (or \emph{maximal}) if whenever $\pi$ is a representation dilating $\rho$, then $\pi = \rho \oplus \psi$ for some representation $\psi$.
		\end{enumerate}
	\end{Def}
	
	Using the definitions above, it was shown in \cite[Proposition 2.4]{DS18} that dilation extremity and the UEP are equivalent, and that a representation is dilation extreme if and only if its unitization is dilation extreme \cite[Proposition 2.5]{DS18}. The $C^*$-envelope $C^*_e(\A)$ of a non-unital algebra $\A$ coincides with the $C^*$-algebra generated by $\A$ inside $C^*_e(\A^1)$, and the theorem of Dritschel and McCullough (see \cite{DM05}) holds in the possibly-non-unital context. That is, every representation of an operator algebra dilates to a dilation extreme representation. Hence, even in the non-unital setting we have that $C^*_e(\A)$ is the $C^*$-algebra generated by the image of any dilation extreme completely isometric representation. Hence, whether $\A$ is unital or not, the $C^*$-envelope $C^*_e(\A)$ of $\A$ coincides with the universal $C^*$-algebra generated by dilation extreme representations of $\A$.
	
One of the most important properties of dilation extreme representations of operator algebras is their invariance under completely isometric isomorphisms. More precisely, if $\theta :\A \rightarrow \B$ is a completely isometric isomorphism, then a representation $\rho : \B \rightarrow B(\H)$ is dilation extreme for $\B$ if and only if $\rho \circ \theta$ is dilation extreme for $\A$. We will require the following weak versions of dilation extremity which were originally defined in the work of Muhly and Solel \cite{MS98} in the language of Hilbert modules.
	
	\begin{Def}
		Let $\A$ be an operator algebra, and $\rho : \A \rightarrow B(\H)$ a representation. We say that a dilation $\pi: \A \rightarrow B(\K)$ of $\rho$ is
		\begin{enumerate}
			\item an \emph{e-dilation} if $\H$ is \emph{invariant} for $\pi(\A)$.
			
			\item a \emph{c-dilation} if $\H$ is \emph{co-invariant} for $\pi(\A)$.
			
		\end{enumerate}
	\end{Def}
	
	\begin{Def}
		Let $\A$ be an operator algebra, and $\rho : \A \rightarrow B(\H)$ a representation. We say that $\rho$ is
		\begin{enumerate}
			\item \emph{e-dilation extreme} if whenever $\pi : \A \rightarrow B(\K)$ is e-dilation of $\rho$, then in fact $\pi = \rho \oplus \psi$ for some representation $\psi$.
			
			\item \emph{c-dilation extreme} if whenever $\pi : \A \rightarrow B(\K)$ is c-dilation of $\rho$, then in fact $\pi = \rho \oplus \psi$ for some representation $\psi$.
			
		\end{enumerate}
	\end{Def}
	
	The above notions were called ``extension extreme" and ``coextension extreme" by Davidson and Katsoulis in \cite{DK11b}. By a theorem of Sarason \cite[Exercise 7.6]{Paulsen}, it follows that $\rho$ is dilation extreme if and only if it is both e-dilation extreme and c-dilation extreme. In fact, it follows from the work of Dritchel and McCullough that every representation of an operator algebra admits either an e-dilation that is e-dilation extreme or a c-dilation that is c-dilation extreme. The following is the analogue of Arveson's ``invariance of UEP" for c-dilation and e-dilation extremal representations.
	
	\begin{Thm} \label{t:ec-dil}
		Let $\A$ and $\B$ be operator algebras, let $\theta : \A \rightarrow \B$ an isomorphism, and let $\rho : \B \rightarrow B(\H)$ be a representation. Then 
		\begin{enumerate}
			
			\item $\rho$ is e-dilation extreme if and only if $\rho \circ \theta$ is e-dilation extreme.
			
			\item $\rho$ is c-dilation extreme if and only if $\rho \circ \theta$ is c-dilation extreme.
			
		\end{enumerate}
	\end{Thm}
	
	\begin{proof}
		An inverse of a completely isometric isomorphism between operator algebras is again a completely isometric isomorphism. Hence, it is enough to prove one direction in each claim. We will show the forward direction for (2), and the proof for the forward direction of (1) is similar. Assume $\rho$ is c-dilation extreme and let $\pi:\A\to B(\K)$ be a c-dilation of $\rho\circ \theta$.
		Then $\rho(b)=P_{\H}\pi(\theta^{-1}(b))|_{\H}$ for all $b\in \B$. So, $\pi\circ\theta^{-1}:\B\to B(\K)$ is a c-dilation of $\rho$ as $\H$ is coinvariant for $\pi \circ \theta^{-1}(\B) = \pi(\A)$. By our assumption, $\pi\circ\theta^{-1}=\rho\oplus \psi$ for some representation $\psi$. Thus $\pi=(\rho\circ\theta) \oplus (\psi\circ \theta)$, as required. 
	\end{proof}
	
Throughout the paper, we shall denote by $\A \otimes \B$ the spatial tensor product of the operator algebras $\A$ and $\B$ as defined in \cite[Subsection 2.2.2]{BL04}. When $\A$ and $\B$ are both C*-algebras, $\otimes$ coincides with the minimal tensor product of C*-algebras.
	
\begin{Lemma} \label{l:tensor-dil-ext}
Let $\A$ be an operator algebra generating a $C^*$-algebra $\C$ and $\B$ a $C^*$-algebra. Assume  $\tau : \A \rightarrow B(\H)$ is a representation, and $\kappa : \B \rightarrow B(\K)$ a *-representation. Then $\tau$ is dilation extreme if and only if $\tau \otimes \kappa$ is dilation extreme.
\end{Lemma}

\begin{proof}
We may assume, perhaps after unitization, that $\A$ and $\B$ are both unital, and that $\tau$ and $\kappa$ are unital. It will suffice to show that $\tau$ has the unique extension property if and only if $\tau \otimes \kappa$ does. 

So assume $\tau$ has the UEP. Let $\pi : \C \otimes \B \rightarrow B(\H \otimes \K)$ be a unital completely positive extension of $\tau \otimes \kappa$. Then $\pi|_{\C \otimes I}$ is (up to multiplicity) a unital completely positive extension of $\tau$, so that $\pi$ is multiplicative on $\C$. Since $\pi|_{\B}$ is (up to multiplicity) just $\kappa$, we see that $\pi$ is also multiplicative on $\B$. Thus, by \cite[Theorem 3.18]{Paulsen} we have that $\pi$ is multiplicative, and $\tau \otimes \kappa$ has the unique extension property.

Conversely, if $\tau \otimes \kappa$ has the UEP, let $\pi : \C \rightarrow B(\H)$ be a unital completely positive extension of $\tau$. Then clearly $\pi \otimes \kappa$ is a unital completely positive extension of $\tau \otimes \kappa$, so that $\pi \otimes \kappa$ is multiplicative. In particular, $\pi$ is multiplicative, so that $\tau$ has the UEP. 
\end{proof}

\begin{Cor} \label{c:envelope-commutes}
Let $\A$ be an operator algebra and $\B$ a $C^*$-algebra. Then $C^*_e(\A \otimes \B) \cong C_e^*(\A) \otimes \B$.
\end{Cor}

\begin{proof}
Let $\tau : \A \rightarrow B(\H)$ be a dilation extreme completely isometric representation, and $\kappa : \B \rightarrow B(\K)$ an injective $*$-representation. By injectivity of minimal tensor product, we see that $\tau \otimes \kappa$ is completely isometric and by Lemma \ref{l:tensor-dil-ext} we get that $\tau \otimes \kappa$ is dilation extreme. Thus, since the $C^*$-algebra generated by the image of $\tau \otimes \kappa$ is $C^*_e(\A) \otimes \B$, and as any $C^*$-algebra generated by the image of a dilation extreme completely isometric representation is the $C^*$-envelope, we get that $C^*_e(\A \otimes \B) \cong C^*_e(\A) \otimes \B$.
\end{proof}
	
	
	\subsection{Toeplitz and Tensor algebras of $C^*$-correspondences}\label{ss: Toeplitz and Tensor algebras of C*-correspondences}

For the basic theory of Hilbert C*-modules and C*-correspondences we recommend \cite{Lan95, MT05}. For C*-correspondences $E$ and $F$ over a C*-algebra $\A$, we can form the \emph{interior tensor product} $E\otimes_{\A} F$ of $E$ and $F$ as follows. Let $E\otimes_{\alg} F$ denote the quotient of the algebraic tensor product, by the subspace generated by elements of the form:
	$$
	xa\otimes y-x\otimes \phi_F(a)y, \ \ \text{for} \ \  x\in E, y\in F, a\in \A.
	$$
	Define an $\A$-valued inner product, left and right $\A$-actions by:
	$$ 
	\langle x_1\otimes y_1, x_2\otimes y_2\rangle = \langle y_1, \phi_F(\langle x_1,x_2\rangle)y_2\rangle, \ \text{for} \ x_1,x_2\in E, y_1,y_2\in F
	$$
	$$
	(x\otimes y)a=x\otimes (ya), \ \ \text{for} \ x\in E, y\in F,
	$$
	$$
	\phi_{E\otimes F}(a)(x\otimes y) = (\phi_E(a)x)\otimes y, \ \ \text{for} \ x\in E, y\in F, a\in \A.
	$$
	$E\otimes_{\A} F$ is the completion of $E\otimes_{\operatorname{alg}} F$ with respect to the $\A$-valued semi-inner product defined above. One checks that $E\otimes_{\A} F$ is a C*-correspondence over $\A$. We will often abuse notation and write $E \otimes F$ for $E \otimes_{\A} F$ when the context is clear.
	
	\begin{Hypo}
 We assume throughout the paper that every C*-corres\-pondence $E$ over $\A$ is non-degenerate in the sense that $\phi_E(\A)E=E$.
	\end{Hypo}	
	
	Let $E$ be a C*-correspondence over $\A$. Set $E^{\otimes 0 }=\A$ and $E^{\otimes n}=E\otimes E \otimes \cdots\otimes E$, for the $n$-fold tensor product when $n>0$. Notice that we have natural isomorphisms $E^{\otimes n}\otimes  E^{\otimes m}\cong E^{\otimes {(n+m)}}$ for $n> 0$ and $m\geq 0$. Note also that $E^{\otimes 0}\otimes_\A E^{\otimes n}\cong E^{\otimes n}$ for $n\geq 0$ by non-degeneracy.	There is a special injective representation for $\T(E)$ called the Fock space representation of $E$. We denote $\F(E)$ the C*-correspondence over $\A$ defined by $\F(E)=\oplus_{n =0}^{\infty} E^{\otimes n}$. Since $\phi_{\F(E)}$ acts on the $0$-th summand by left multiplication, it is clear that $\phi_{\F(E)}$ is injective, so we will often identify $\A$ as a subalgebra of $\L(\F(E))$ via $\phi_{\F(E)}$. We define Fock space representation $(\iota,T)$ as follows. We let $\iota := \phi_{\F(E)}$, and for each $\xi\in E$ we set $T_{\xi} \in \L(\F(E))$ by
	$$
	T_{\xi} (a)=\xi a \ \ \text{ and }\ \ T_{\xi}(\xi_1\otimes\ldots\otimes \xi_n)=\xi\otimes \xi_1\otimes\ldots\otimes \xi_n
	$$
	for $\ a\in \A, \ \xi,\xi_1,\ldots\xi_n\in E$. For each $n\in \bN$, denote by $T^{(n)}: E^{\otimes n}\to \L(\F(E))$ the map uniquely determined on simple tensors by $T^{(n)}_{\xi_1 \otimes \ldots \otimes \xi_n} = T_{\xi_1} \circ \ldots \circ T_{\xi_n}$. When we have $\xi \in E^{\otimes n}$ we will occasionally abuse notations and write $T_{\xi}$ to mean $T^{(n)}_{\xi}$ and the degree will be clear from context.
	
	By \cite[Theorem 2.12]{MS98} the representation $(\iota,T)$ is universal in the sense that $\T(E) \cong C^*(\iota, T)$. Similarly, by \cite[Theorem 3.10]{MS98} the tensor algebra of $E$ coincides with $\overline{\Alg}(\iota,T)$, the norm-closed operator algebra generated by Fock creation operators. Hence we identify $\T_+(E)$ as the norm-closed operator subalgebra of $\T(E)$ generated by the image of $\iota$ and $T$.

	We may also use Fock space representation to obtain another description of the ideal of relations $\J_E$ generated by $\{ \psi_T(\phi_E(a))-\iota(a): a\in J_E \}$ which yields the Cuntz-Pimsner algebra $\O(E)$ of $E$. Let $\sigma: \L(\F(E))\to \L(\F(E))/\K(\F(E)J_E)$ be the canonical quotient map. By \cite[Proposition 6.5]{Kat04b} we get that $\O(E)\cong C^*(\sigma\circ \iota, \sigma\circ T)$. Thus, $\O(E)\cong \T(E)/\K(\F(E)J_E)$ so that $\J_E$ is identified with $\K(\F(E)J_E)$ and we can describe $\J_E$ as 
	$$
	\overline{\Span}\{T_{\xi} a T^*_{\eta} : \xi\in E^{\otimes n}, \ \eta \in E^{\otimes m}, \ a\in J_E, \ n,m\in \bN \}.
	$$
	
	\begin{Def}
	Let $E$ and $F$ be C*-correspondences over $\A$ and $\B$, respectively. We say that $E$ and $F$ are \textit{unitarily isomorphic} if there exist a surjective, isometric map $U: E\to F$ and a *-isomorphism $\rho: \A\to \B$, s.t. $U(a\cdot \xi\cdot b)=\rho(a)\cdot U(\xi)\cdot \rho(b)$ for all $a,b\in \A, \xi \in E$.
	\end{Def}
	
See \cite[Subsection 2.1]{DO18} for more on isomorphisms of C*-corres\-pondences. For a C*-correspondence over $\A$, we denote by $P_0\in \L(\F(E))$ the orthogonal projection onto $E^{\otimes 0} = \A$, and we let $\Psi : \T(E)\to \A$ be the compression given by $\Psi(x)=P_0xP_0$, for all $x\in \T(E)$. Notice that for $a\in \A$, if $\Psi(a)=0$ then $a=0$. The following folklore result is a strengthening of \cite[Lemma 4.6.24]{BO} which we obtain by using C*-envelope techniques.

\begin{Prop}\label{P: KatIdealPres}
	Let $E$ be a C*-correspondence over $\A$ and let $\B$ be any C*-algebra. Then $(\iota \otimes \id_{\B}, T\otimes \id_{\B})$ is a rigged representation which induces an isomorphism $\rho : \T(E\otimes \B) \rightarrow \T(E)\otimes\B$. The isomorphism $\rho$ maps $\T_+(E \otimes \B)$ onto $\T_+(E) \otimes \B$ and induces an isomorphism between $\O(E \otimes \B)$ and $\O(E) \otimes \B$. In particular, $\rho$ maps $\J_{E\otimes \B}$ to $\J_E \otimes \B$.
\end{Prop}

\begin{proof}
It is standard to verify that $(\iota\otimes \id_\B, T\otimes\id_\B)$ is a rigged representation of the C*-correspondence $E\otimes \B$ over $\A \otimes \B$.
	By universality of $\T(E\otimes \B)$, we have a surjection $\rho: \T(E\otimes\B)\to \T(E)\otimes \B$ given by $\rho(T_{\xi \otimes b}) = T_{\xi} \otimes b$ for $\xi \in E$ and $b\in \B$. Since $\T(E)\otimes \B$ admits the gauge action $z \mapsto \gamma_z\otimes \id_\B$, we know from theorem \cite[Theorem 6.2]{Kat04a} that $\rho$ is an isomorphism if and only if
	\[
	I_{(\iota\otimes \id_\B, T\otimes\id_\B)}'\coloneqq\{ \ x\in \A\otimes\B \ | \ (\iota\otimes\id_\B)(x)\in B_1 \ \}
	\]
	is trivial, where 
	\[
	B_1=\overline{\Span} \{ \ (T\otimes \id_\B)(x)(T\otimes\id_\B)(y)^* \ | \ x,y\in E\otimes \B \ \}. 
	\]
So we show that $I_{(\iota\otimes \id_\B, T\otimes\id_\B)}'$ is trivial.
	Indeed, let $x \in (\A \otimes \B) \cap B_1$. We clearly have that $(\Psi \otimes\id_\B)(x)=0 $ as $x\in B_1$, while $(\Psi \otimes\id_\B)(x) = x$ as $x \in \A \otimes \B$. Thus, $x = 0$ and $\rho$ is an isomorphism.

By injectivity of the minimal tensor product of operator algebras, it follows that the natural isomorphism $\T(E)\otimes \B\cong \T(E\otimes \B)$ restricts to an isomorphism of operator algebras $\T_+(E)\otimes \B\cong \T_+(E\otimes \B)$. Thus, by Corollary \ref{c:envelope-commutes} and \cite{KK04} we have the following chain of isomorphisms
$$
\O(E \otimes \B) \cong C^*_e(\T_+(E \otimes \B)) \cong C^*_e(\T_+(E) \otimes \B) 
$$
$$
\cong C^*_e(\T_+(E)) \otimes \B \cong \O(E) \otimes \B.
$$
This yields an isomorphism $\widetilde{\rho} : \O(E \otimes \B) \rightarrow \O(E) \otimes \B$ which sends $S_{\xi \otimes b}$ to $S_{\xi} \otimes b$ for every $\xi \in E$ and $b \in \B$ where $(\iota,S)$ is a rigged covariant representation such that $\O(E) \cong C^*(\iota,S)$. Let $q_E \otimes \id_{\B} : \T(E) \otimes \B \rightarrow \O(E) \otimes \B$ and $q_{E \otimes \B} : \T(E \otimes \B) \rightarrow \O(E \otimes \B)$ denote the canonical quotient maps. Since this occurs on generators, we see that $(q_E \otimes \id_{\B}) \circ \rho = \widetilde{\rho} \circ q_{E \otimes \B}$. Hence, it follows that $\J_{E\otimes \B}$ is identified with $\J_E \otimes \B$ via $\rho$.
\end{proof}

\begin{Cor} \label{c:tensor-katsura-preserved}
Let $E$ be a C*-correspondence over $\A$, and let $\B$ be an exact C*-algebra. Then the natural isomorphism $\rho: \T(E\otimes \B) \rightarrow \T(E)\otimes\B$ of Proposition \ref{P: KatIdealPres} maps $J_{E\otimes\B}$ to $J_E\otimes \B$.
\end{Cor}

\begin{proof}
First note that since $\phi_{E \otimes \B} = \phi_E \otimes \id_{\B}$ under the identification $\K(E \otimes \B) \cong  \K(E) \otimes \B$, we see that $\ker (\phi_{E\otimes \B}) = \ker (\phi_E) \otimes \B$.

Next, if $a \in J_E$, then $\phi_E(a) \in \K(E)$ and $ac = 0$ for all $c \in \ker \phi_E$. Thus, for any $b \in \B$ we have that $\phi_{E \otimes \B}(a \otimes b) = \phi(a) \otimes b \in \K(E) \otimes \B \cong \K(E \otimes \B)$. Furthermore, by exactness of $\B$ we get that $\ker (\phi_E \otimes \id_{\B}) = \ker(\phi_E) \otimes \B$. Thus, by verifying this on simple tensors we get for any $c \in \ker (\phi_E \otimes \id_{\B})$ that $(a \otimes b) c = 0$. Hence, we see that $J_E \otimes \B \subseteq J_{E \otimes \B}$. 

Conversely, again by exactness of $\B$, we get the following short exact sequence
$$
\begin{tikzcd}
	0 \arrow[r] & \J_E \otimes \B \arrow[r]  & \T(E) \otimes B \arrow[r] & \O(E) \otimes \B  \arrow[r] & 0
	\end{tikzcd}.
$$
Thus, $\J_E \otimes \B$ is the ideal generated by
$$
\{ \ (\psi_{T^E} \otimes \id_{\B})((\phi_E \otimes \id_{\B})(c)) - (\iota^E \otimes \id_{\B})(c) \ | \ c \in J_E \otimes \B \ \}.
$$
From Proposition \ref{P: KatIdealPres}, as $\J_E \otimes \B$ is corresponds bijectively to $\J_{E \otimes \B}$ via $\rho$, we get that $\J_{E \otimes \B}$ is the ideal generated by 
$$
\{ \ \psi_{T^{E \otimes \B}}(\phi_{E \otimes \B}(c)) - (\iota^{E\otimes \B})(c) \ | \ c \in J_{E} \otimes \B \ \}.
$$
Hence, we see that the surjection from the relative Cuntz-Pimsner algebra $\O(E \otimes \B, J_E \otimes \B)$ (See \cite[Section 11]{Kat07}) onto $\O(E \otimes \B)$ is injective. From \cite[Corollary 11.8]{Kat07} we deduce that $J_E \otimes \B = J_{E \otimes \B}$.
 
\end{proof}

	
\section{Hierarchy of isomorphism problems}\label{S: Hierarchy of isomorphism problems}
	
In this section we establish the aforementioned hierarchy between different notions of isomorphisms of Toeplitz and tensor algebras.
	
	\begin{Thm} \label{T:graded-to-ses}
		Let $E$ and $F$ be $C^*$-correspondences over $\A$ and $\B$ respectively. Suppose that $\varphi : \T(E) \rightarrow \T(F)$ is a base-preserving graded *-isomorphism. Then $\varphi$ is a base-preserving s.e.s. isomorphism.
	\end{Thm}
	
	\begin{proof}
		Suppose $\varphi : \T(E) \rightarrow \T(F)$ is a base-preserving graded *-isomor\-phism. Let $q_E : \T(E) \rightarrow \O(E)$ and $q_F : \T(F) \rightarrow \O(F)$ be the canonical quotient maps. Since $\varphi$ is graded, it follows by the discussion after Definition \ref{d:GradIso} that $\varphi$ is equivariant.
		
		The map $\psi\coloneqq q_E\circ \varphi^{-1}: \T(F)\rightarrow \O(E)$ then restricts to an injective rigged representation $(\psi\circ \iota_B, \psi\circ T^F)$ which admits a gauge action, where $(\iota_B, T^F)$ denotes the Fock representation. Notice that $C^*(\psi\circ \iota_B,\psi\circ T^F)=\O(E)$. 
		By \cite[Proposition 7.14]{Kat07} there is an induced surjective *-homomorphism $\rho : \O(E) \rightarrow \O(F)$ such that $\rho \circ \psi = q_F$, namely $\rho\circ q_E=q_F\circ \varphi$. This implies that $\varphi( \K(\F(E)J_E))\subseteq   \K(\F(F)J_F)$. The symmetric argument with $\varphi$ instead of $\varphi^{-1}$ then shows the reversed inclusion.
	\end{proof}
	
	Let $F$ be a $C^*$-correspondence over a $C^*$-algebra $\B$ and assume that $\rho: \A\to \B$ is a *-isomorphism. Then $F$ can be realized naturally, as a $C^*$-correspondence $F_{\rho}$ over $\A$, via $\rho$.
	The new operations are given by 
	$ {\langle \xi,\eta \rangle}_{\rho} \coloneqq \rho^{-1}({\langle \xi,\eta \rangle}_\A), \text { for } \xi,\eta\in E; \ a \cdot \xi= \rho(a)\cdot \xi; \text{ and } \xi\cdot a \coloneqq \xi \cdot \rho(a), \text{ for all } \xi\in F \text{ and } a\in \A$. The identity map $U:F \to F_{\rho}$ is then a unitary isomorphism. 
	
	This leads to the following simple reduction for base-preserving isomorphisms. Suppose $F$ is a $C^*$-correspondence over $\A$ and $\rho : \A \rightarrow \B$ is a *-isomorphism. Then we have that
	\begin{enumerate}
		\item
		$\T_+(F)$ and $\T_+(F_{\rho})$ are base-preserving graded isomorphic.
		\item
		$\T(F)$ and $\T(F_{\rho})$ are base-preserving graded isomorphic.
		\item
		$\T(F)$ and $\T(F_{\rho})$ are base-preserving s.e.s. isomorphic.
	\end{enumerate}
	Hence, when $E$ and $F$ are $C^*$-correspondences over $\A$ and $\B$ respectively, and $\phi$ is a base-preserving isomorphism from an algebra of $E$ to an algebra of $F$, after composing with one of the isomorphisms above, we may assume that $\phi$ is the identity on the base algebra. For more details, we refer the reader to the discussion after \cite[Definition 2.1]{DO18}. We next recall some definitions and results from Muhly and Solel \cite{MS98}. The following is \cite[Definition 3.1]{MS98}. 
	
	\begin{Def} \label{D:rig-dil}
		Let $(\pi, t)$ be a completely contractive representation of a $C^*$-correspondence $E$ over a $C^*$-algebra $\A$ on a Hilbert space $\H$.
		A \textit{rigged dilation} of $(\pi,t)$ is a rigged representation $(\sigma,s)$ of $E$ on a Hilbert space $\K\supseteq \H$, s.t.
		\begin{enumerate}
			\item $\sigma$ dilates $\pi$, i.e. $\pi(a)=P_{\H}\sigma(a)|_{\H}$, for all $a\in \A$;
			\item $s$ dilates $t$, i.e. $t(\xi)=P_{\H}s(\xi)|_{\H}$, for all $\xi\in E$; and
			\item $\H$ is co-invariant under each $s(\xi)$, $\xi\in E$.
		\end{enumerate}
	\end{Def}
	Since $\sigma$ is a $*$-homomorphism that dilates $\pi$, it is easy to see that $\sigma$ must be of the form $\pi\oplus\psi$ for some $*$-representation $\psi$. In \cite[Theorem 3.3]{MS98} Muhly and Solel provided a generalization of Nagy--Foias dilation to representations of $C^*$-correspondences. This theorem shows that every completely contractive representation of $E$ admits a rigged dilation. As a consequence, we have a one-to-one correspondence between completely contractive representations of $E$ and representations of $\T_+(E)$ as shown in \cite[Theorem 3.10]{MS98}. More precisely, to every completely contractive representation $(\pi,t)$ of a correspondence $E$ over a $C^*$-algebra $\A$, there is a unique completely contractive representation $\rho$ of $\T_+(E)$, satisfying:
		\begin{enumerate}
			\item $\rho(T_{\xi})=t(\xi)$, for $\xi\in E$; and
			\item $\rho(a)=\pi(a)$, for $a\in \A$.
		\end{enumerate}
The map $(\pi,t)\mapsto \rho$ is then bijective onto the all representations of $\T_+(E)$.

	Thus, for a representation $\rho: \T_+(E)\to B(\H)$ let $(\pi,t)$ be the associated completely contractive representation of $E$ given by $\pi = \rho \circ \iota$ and $t = \rho \circ T$. Then $(\pi,t)$ has a rigged dilation $(\sigma,s)$. Thus, $(\sigma, s)$ induces a representation $\widetilde{\rho} :\T_+(E)\to B(\K)$ such that $\widetilde{\rho}(a)=\sigma(a)$ for $a\in \A$, and $\widetilde{\rho}(T_{\xi})=s(\xi)$ for $\xi\in E$. It follows by Definition \ref{D:rig-dil} and the succeeding discussion that $\H$ is co-invariant for $\widetilde{\rho}$. That is, $\widetilde{\rho}$ is a c-dilation of $\rho$. The following can be obtained by combining \cite[Proposition 4.2]{MS98} and \cite[Corollary 4.7]{MS98} using the language of orthoprojective modules. For the sake of posterity we provide a direct proof.
	
	\begin{Prop}\label{p:c-dil-rigged}
		Let $E$ be a $C^*$-correspondence over a $C^*$-algebra $\A$, and let $\rho:\T_+(E)\to B(\H)$ be a representation. Then $\rho$ is c-dilation extreme if and only if $(\rho\circ\iota,\rho\circ T)$ is a rigged representation of $E$.
	\end{Prop}
	
	\begin{proof}
		$(\Rightarrow):$ Let $\widetilde{\rho}: \T_+(E)\to B(\K)$ be the completely contractive representation described in the above paragraph. As $\widetilde{\rho}$ is a c-dilation of $\rho$ and $\rho$ is c-dilation extreme, $\widetilde{\rho}=\rho\oplus \psi$ for some representation $\psi$. Moreover, we know that $\widetilde{\rho}$ is induced by a rigged representation, $(\sigma, s)=(\widetilde{\rho}\circ\iota, \widetilde{\rho}\circ T)$ of $E$. Thus, $(\rho\circ \iota,\rho\circ T)$ is also a rigged representation of $E$.
		
		$(\Leftarrow):$ 
		Assume $(\rho\circ \iota,\rho\circ T)$ is a rigged representation of $E$. We want to show that $\rho$ is c-dilation extreme, so let $\widetilde{\rho}:\T_+(E)\to B(\K)$ be a c-dilation of $\rho$.
		By the above discussion, there exists a c-dilation of $\widetilde{\rho}$, denoted $\widehat{\rho}: \T_+(E)\to B(\widehat{\K})$, s.t. $(\widehat{\rho} \circ\iota, \widehat{\rho} \circ T)$ is a rigged representation of $E$. Viewing $\H\subseteq \K \subseteq \widehat{\K}$, one can check that $\widehat{\rho}$ is also a c-dilation of $\rho$. So without loss of generality we assume that $(\sigma,s) = (\widetilde{\rho} \circ \iota, \widetilde{\rho} \circ T)$ is rigged.
		
		We claim that $\widetilde{\rho}=\rho\oplus \psi$, for some representation $\psi$. Indeed, view $\widetilde{\rho}$ as a $2\times 2$ block matrix acting on $\H\oplus \H^{\perp}$. For $\xi\in E$, as $\H$ is co-invariant for $\widetilde{\rho}(\T_+(E))$ and $\rho$ is a rigged representation, we may write
		$$
		\widetilde{\rho}(T_{\xi}) = \begin{bmatrix} 
		\rho(T_\xi)&& 0\\
		X_{\xi} && Y_{\xi}\\
		\end{bmatrix}
		$$
		So that,
		$$
		\widetilde{\rho}(T_{\xi})^*\widetilde{\rho}(T_{\xi}) = 
		\begin{bmatrix} 
		\rho(\langle \xi,\xi\rangle)+X_{\xi}^*X_{\xi}&& * \\
		* && * \\
		\end{bmatrix}.
		$$
On the other hand, using that $(\widetilde{\rho} \circ \iota,\widetilde{\rho}\circ T)$ is a rigged representation of $E$ we get that,
		$$
		\widetilde{\rho}(T_{\xi})^*\widetilde{\rho}(T_{\xi})=\widetilde{\rho}(\langle \xi, \xi\rangle)=  \begin{bmatrix} 
		\rho(\langle \xi, \xi\rangle)&& 0\\
		* && *\\
		\end{bmatrix}.
		$$
		Combining both equations, we get $X_{\xi}^*X_{\xi}=0$, and so $X_{\xi}=0$. Therefore, for all $\xi\in E$, one has 
		$$
		\widetilde{\rho}(T_{\xi})=\begin{bmatrix} 
		\rho(T_\xi)&& 0\\
		0 && Y_{\xi}\\
		\end{bmatrix}.
		$$
		Moreover, since $\widetilde{\rho} \circ\iota$ is a *-homomorphism that dilates the *-representation $\rho\circ\iota$, we know that $\rho\circ\iota$ is a direct summand of $\widetilde{\rho} \circ \iota$. Hence, as $\H$ is reducing for $\widetilde{\rho}(a)$ and $\widetilde{\rho}(T_{\xi})$ for every $a \in \A$ and $\xi \in E$, and as the image of $\T_+(E)$ under $\widetilde{\rho}$ is generated by such elements, we see that $\widetilde{\rho}$ has $\H$ as a reducing subspace. Therefore, $\widetilde{\rho}$ must be of the form $\widetilde{\rho}=\rho\oplus\psi$, for some representation $\psi$. This shows that $\rho$ is c-dilation extreme.
	\end{proof}
	
	\begin{Thm} \label{T:non-sa-to-sa}
		Let $E$ and $F$ be $C^*$-correspondences over $C^*$-algebras $\A$ and $\B$, respectively. Let $\phi: \T_+(E)\to\T_+(F)$ be a completely isometric isomorphism. Then $\phi$ extends to a base-preserving s.e.s. isomorphism $\varphi$ of $\T(E)$ and $\T(F)$. Furthermore, if $\phi$ is graded, then $\varphi$ is also graded.
	\end{Thm}
	
	\begin{proof}
		If $\varphi$ extends a graded isomorphism $\phi : \T_+(E) \rightarrow \T_+(F)$ to a *-isomorphism between $\T(E)$ and $\T(F)$, it is easy to see on *-monomials that $\varphi$ is also graded.
		
		Now assume $\phi: \T_+(E)\to\T_+(F)$ is a completely isometric isomorphism. We use Meyer's theorem \cite{Mey01} to extend, if necessary, to a unital complete isometry $\phi^1 : \T_+(E)^1 \rightarrow \T_+(F)^1$. Then \cite[Proposition 2.12]{Paulsen} shows that $\phi^1$ extends to a unital complete isometry $\widetilde{\phi}$ between $(\T_+(E)^1)^* + \T_+(E)^1$ and $(\T_+(F)^1)^* + \T_+(F)^1$. We then notice that $\Delta(\T_+(E)) := \{ \ T \in \T_+(E) \ | \  T^* \in \T_+(E) \ \}$ must equal the base algebra $\A$, and similarly we have that $\Delta(\T_+(F)) = \B$. Since $\widetilde{\phi}$ preserves involution and sends $\T_+(E)$ to $\T_+(F)$, it must map $\Delta(\T_+(E))$ to $\Delta(\T_+(F))$. In other words, $\phi$ must be base-preserving. It is then clear that any extension of $\phi$ will be base-preserving as well.
		
		Next, we show that $\phi$ extends to a s.e.s. isomorphism. We let $\{\rho^E_i\}_{i\in I}$ be all c-dilation extreme representations for $\T_+(E)$ up to unitary equivalence and for sufficiently large Hilbert space. Then, we set $\rho^F_i:= \rho^E_i \circ \phi^{-1}$ so that by Theorem \ref{t:ec-dil} we have that $\{\rho^F_i\}_{i\in I}$ are all c-dilation extreme representations of $\T_+(F)$ up to unitary equivalence. By Proposition \ref{p:c-dil-rigged}, we have that $C^*(\oplus_{i\in I}\rho_i^E)$ and $C^*(\oplus_{i\in I}\rho_i^F)$ are universal with respect to rigged representations of $E$ and $F$, respectively. Hence, that there are *-isomorphisms $\pi_E: C^*(\oplus_{i\in I}\rho_i^E) \rightarrow \T(E)$ and $\pi_F: C^*(\oplus_{i\in I}\rho_i^F) \rightarrow \T(F)$. Thus, since by construction $C^*(\oplus_{i\in I}\rho_i^E)= C^*(\oplus_{i\in I}\rho_i^F)$, a straightforward verification shows that $\varphi:= \pi_F \circ \pi_E^{-1}$ is a *-isomorphism between $\T(E)$ and $\T(F)$ which extends $\phi$.
		
		Next, let $\{\rho^E_i\}_{i\in J}$ be those c-dilation extreme representations of $\T_+(E)$ which are dilation extreme. By invariance of dilation extreme representations we get that $\{\rho^F_i\}_{i\in J}$ are all dilation extreme representations of $\T_+(F)$ up to unitary equivalence. By (the possibly non-unital version of) Dritchel and McCullough \cite{DM05} we have $C^*(\oplus_{i\in J}\rho_i^E) = C^*_e(\T_+(E))$ and $C^*(\oplus_{i\in J}\rho_i^F) = C^*_e(\T_+(F))$. Furthermore, by \cite{KK06} we know that $C^*_e(\T_+(E)) \cong \O(E)$ and $C^*_e(\T_+(F)) \cong \O(F)$. Hence, there *-isomorphisms $\tau_E : C^*_e(\T_+(E)) \rightarrow \O(E)$ and $\tau_F : C^*_e(\T_+(F)) \rightarrow \O(F)$. As before, $\widetilde{\varphi}:= \tau_F \circ \tau_E^{-1}$ is a *-isomorphism between $\O(E)$ and $\O(F)$ which extends $\phi$.
		
	Furthermore, by construction we have that the natural quotient maps $q_E : C^*(\oplus_{i\in I}\rho_i^E) \rightarrow C^*_e(\T_+(E))$ and $q_F : C^*(\oplus_{i\in I}\rho_i^F) \rightarrow C^*_e(\T_+(F))$ are equal, so we get that the identified quotient maps $q_E : \T(E) \rightarrow \O(E)$ and $q_F : \T(F) \rightarrow \O(F)$ satisfy $q_F \circ \varphi = \widetilde{\varphi} \circ q_E$. Hence, we see that $\varphi$ is a base-preserving s.e.s. isomorphism.
	\end{proof}
	
	To conclude, we have obtained the following hierarchy of isomorphism problems for $C^*$-correspondences, tensor algebras and Toeplitz algebras.
	
	\begin{Cor} \label{C:main-square}
		Let $E$ and $F$ be $C^*$-correspondences over $C^*$-algebras $\A$ and $\B$, respectively. Consider the following:
		\begin{enumerate}
			\item
			$E$ and $F$ are unitarily isomorphic $C^*$-correspondences.
			
			\item
			$\T_+(E)$ and $\T_+(F)$ are graded completely isometrically isomorphic.
			\item
			$\T_+(E)$ and $\T_+(F)$ are completely isometrically isomorphic.
			\item
			$\T(E)$ and $\T(F)$ are base-preserving graded isomorphic.
			\item
			$\T(E)$ and $\T(F)$ are base-preserving s.e.s. isomorphic.
		\end{enumerate}
		Then $(1)$ implies $(2)$, $(2)$ implies $(3)$ and $(4)$, and each of $(3)$ and $(4)$ separately imply $(5)$.
	\end{Cor}
	
	\begin{proof}
		Clearly $(1)$ implies $(2)$ and $(2) \implies (3)$. By Theorem \ref{T:non-sa-to-sa} we see that $(2) \implies (4)$ and $(3) \implies (5)$. Finally, an application of Theorem \ref{T:graded-to-ses} shows that $(4) \implies (5)$.
	\end{proof}
	
	\begin{Rmk} \label{R:no-completely}
		In case $\phi : \T_+(E) \rightarrow \T_+(F)$ is a graded (not necessarily completely) isometric isomorphism, since $\phi|_{\A} : \A \rightarrow \B$ is an isometric isomorphism it is automatically a *-isomorphism by invoking \cite[Corollary 4.2]{Gar65}. Hence, the proof of \cite[Theorem 4.3 item (2)]{DO18} can be carried out to show that the $C^*$-correspondences $E$ and $F$ are unitarily isomorphic. Thus, items (1) and (2) in the above theorem are both equivalent to the existence of a graded isometric isomorphism between $\T_+(E)$ and $\T_+(F)$. 
	\end{Rmk}

	
	\section{Graded isomorphisms}\label{S: Graded isomorphisms}
	In this section, we investigate graded base-preserving isomorphisms of Toeplitz algebras under additional assumptions on the $C^*$-correspondence. 
	
	Following the notations of \cite{Kat04b}, we recall the construction of core subalgebras of Fock representation and their properties.
	
	\begin{Def} \label{d:cores}
		We define the core C*-subalgebras $B_n\subseteq \T(E)$ by
		$$
		B_0\coloneqq \phi_{\F(E)}(\A), \text{  and  } B_n\coloneqq \overline{\Span}\{ T^{(n)}_{\xi}T^{(n)*}_{\eta}: \xi,\eta \in E^{\otimes n} \}.
		$$
	\end{Def}
	We define $B_{[n,m]} = B_n + \cdots + B_m$ for $m\geq n$. And then set $B_{[n,\infty)}= \overline{\bigcup_{m\geq n} B_{[n,m]}}$. We note that $B_{[n,\infty)}$ is a decreasing sequence of ideals. By \cite[Section 5]{Kat04b} we see that that $B_n$ is identified with $\K(E^{\otimes n})$ by the *-isomorphism $$\K(E^{\otimes n})\to B_n, \ \ \ \ \theta_{\xi, \eta}\mapsto T^{(n)}_{\xi}T^{(n)*}_{\eta}.$$
	
	\begin{Prop} \label{p:ses} We have a short-exact sequence:
		$$0\to B_{[n+1,\infty)}\to B_{[n,\infty)}\to B_{n}\to 0$$ 
		which splits naturally.
	\end{Prop}
	
	\begin{proof}
		We first show that $B_{[n+1,\infty)}\cap B_n=\{0\}$. Let $T\in B_{[n+1,\infty)}\cap B_n$. Observe that $B_n\subseteq \L(E^{\otimes n})$. However, $T(E^{\otimes n})=0$, since $T\in B_{[n+1,\infty)}$. Hence by third isomorphism theorem we get, 
		$$
		\frac{B_{[n,\infty)}}{B_{[n+1,\infty)}}=\frac{B_n+B_{[n+1,\infty)}}{B_{[n+1,\infty)}}\cong \frac{B_n}{B_n\cap B_{[n+1,\infty)}}\cong B_n.
		$$
		Finally, the natural inclusion $B_n\hookrightarrow B_{[n,\infty)}$ gives the splitting map for the short exact sequence.
	\end{proof}
	
	\begin{Prop} \label{closed formula for cores}
		Let $E$ be a $C^*$-correspondence over $\A$. Then for each $n\in \bN$ we have $B_{[n,\infty)}=\T(E)_n \T(E)_n^*$. That is, 
		$$
		B_{[n,\infty)}= \overline{\Span}\{TS^*: T,S\in \T(E)_n \}.
		$$
	\end{Prop}	
	
	\begin{proof}
		$(\subseteq):$ Let $\xi_1\in E^{\otimes k_1},\xi_2\in E^{\otimes k_2},\eta_1\in E^{\otimes l_1},\eta_2\in E^{\otimes l_2}$ be such that $k_1-l_1=n=k_2-l_2$. Then up to approximation by tolerance $\epsilon$, we have 
		$$
		T^{(k_1)}_{\xi_1}T^{(l_1)*}_{\eta_1}T^{(l_2)}_{\eta_2}T^{(k_2)*}_{\xi_2} \sim_{\epsilon} \begin{cases}
		
		T^{(k_1)}_{\xi_1}T^{(l_2 - l_1)}_{\zeta} T^{(k_2)*}_{\xi_2} & \text{if  } l_2\geq l_1 \\
		
		T^{(k_1)}_{\xi_1}T^{(l_1 - l_2)*}_{\zeta'} T^{(k_2)*}_{\xi_2} & \text{if  } l_2<l_1
		\end{cases}
		$$ 
		for some $\zeta \in E^{\otimes (l_2-l_1)}$ and $\zeta' \in E^{\otimes (l_1-l_2)}$. Thus, $T^{(k_1)}_{\xi_1}T^{(l_1)*}_{\eta_1}T^{(l_2)}_{\eta_2}T^{(k_2)*}_{\xi_2} \in B_{k_1}+B_{k_2}\subseteq B_{[n,\infty)}$, as required.

		$(\supseteq):$ It is enough to show that $T^{(n+m)}_{\xi} {T^{(n+m)*}_{\eta}}\in \T(E)_n \T(E)_n^*$, for $\xi, \eta\in E^{\otimes (n+m)}$ and $m\geq 0$. From the grading we get 
		$$
		T^{(n+m)}_{\xi} {T^{(n+m)*}_{\eta}} \in \T(E)_n \T(E)_m \T(E)_m^* \T(E)_n^* \subseteq
		$$
		$$
		\subseteq \ \T(E)_n\T(E)_0\T(E)_n^* \  \subseteq \ \T(E)_n\T(E)_n^* \ ,
		$$
		so we are done.
\end{proof}

	\begin{Cor} \label{c:core-pres}
		Let $E$, $F$ be $C^*$-correspondences over $\A$ and $\B$, respectively. Let $\varphi: \T(E)\to \T(F)$ be a graded *-isomorphism.
		Then $\varphi(B^E_{[n,\infty)})=B^F_{[n,\infty)}$.
	\end{Cor}
	
	For what follows, we denote by $\Psi_n: B_{[n,\infty)}\to {B}_n$ the quotient map by $\B_{[n+1,\infty)}$ as in Proposition \ref{p:ses}.
	
	\begin{Cor}\label{c:core-iso}
		Let $E$, $F$ be $C^*$-correspondences over $\A$ and $\B$, respectively. Let $\varphi: \mathcal{T}(E)\to \mathcal{T}(F)$ be a graded *-isomorphism. Then $\varphi_n\coloneqq \Psi_n \circ \varphi|_{B_n}$ is a *-isomorphism.
	\end{Cor}
	
	\begin{proof}
		By Corollary \ref{c:core-pres} we have that $\varphi$ restricts to an isomorphism between $B^E_{[n,\infty)}$ and $B^F_{[n,\infty)}$ for all $n \in \bN$. By Proposition \ref{p:ses} we get that $\varphi_n$ is the induced isomorphism between the quotient algebras $B^E_n$ and $B^F_n$.
	\end{proof}
	
	When $E$ is a general Hilbert C*-module, by \cite[Page 10]{Lan95} we know that operators in $\K(E)$ may fail to be compact operators as bounded operators the Banach space $E$. Thus, we make a distinction and say that a C*-algebra $\A$ is a \emph{compact operator subalgebra} if $\A$ is a subalgebra of $\K(\H)$ on some Hilbert space $\H$.
	
	\begin{Prop}\label{p:unitary-iso-cpt}
		Let $E$, $F$ be $C^*$-correspondences over $\A$ and $\B$, respectively, such that $\A$ (or $\B$) is a subalgebra of compact operators. Let $\varphi: \T(E)\to \T(F)$ be a base-preserving graded *-isomorphism. Then, there exists a unitary isomorphism $U:E\to F$ implemented by the *-isomorphism $\rho:= \varphi|_{\A}: \A \rightarrow \B$.
	\end{Prop}
	
	\begin{proof}
		As $\A \cong \B$ via $\varphi|_{\A}$, without loss of generality we may assume that $\A = \B$ and that $\varphi|_{\A} = \id_{\A}$. By Corollary \ref{c:core-iso} we have a *-isomorphism $\varphi_1 : B_1^E \rightarrow B_1^F$ induced from $\varphi$ via $\varphi_1 := \Psi_1 \circ \varphi|_{B_1^E}$. After we identify $B^E_1$ and $B^F_1$ with $\K(E)$ and $\K(F)$ respectively, and since $\A$ and $\B$ are subalgebras of compact operators, we may appeal to \cite[Corollary 1]{Asa08} to see that the *-isomorphism $\varphi_1: \K(E)\to \K(F)$ is of the form $\Ad_U$ for some unitary operator $U:E\to F$ such that $U(\xi\cdot a)=U(\xi)\cdot a$ for $\xi\in E, a\in \A$. 
		
		It is left to show that $U(a\cdot \xi)= a \cdot U(\xi)$ for $\xi\in E,a\in A$. Let $\Psi_1: B_{[1,\infty)}\to B_1$ be the quotient map as in Proposition \ref{p:ses}. We first show that $\Psi_1(a \cdot T)= a \cdot \Psi_1(T)$ for all $a\in \A$ and $T \in B_{[1,\infty)}$. Indeed, write $T=S+W$, with $S \in B_1$ and $W \in B_{[2,\infty)}$. Clearly, $a \cdot S \in B_1$ and $a \cdot W \in B_{[2,\infty)}$. Thus, $\Psi_1(a T)=\Psi_1(aS+aW)=a S=a \Psi_1(T)$. Next, let $\xi,\eta\in E$ and $a\in \A$ be given. Then we have
		$$
		\varphi_1(T^{(1)}_{a\cdot \xi}T^{(1)*}_\eta)=T^{(1)}_{U(a\cdot \xi)}T^{(1)*}_{U(\eta)}
		$$
		On the other hand,
		$$
		\varphi_1(T^{(1)}_{a\cdot \xi}T^{(1)*}_\eta)=\Psi_1(\varphi(a \cdot T^{(1)}_{\xi}T^{(1)*}_\eta))=\Psi_1(a \cdot \varphi(T^{(1)}_\xi T^{(1)*}_{\eta}))=
		$$
		$$
		a \cdot \varphi_1(T^{(1)}_{\xi} T^{(1)*}_{\eta})=a \cdot T^{(1)}_{U(\xi)}T^{(1)*}_{U(\eta)}.
		$$
		In particular, it follows that $T^{(1)}_{U(a\cdot \xi)} \langle x,y \rangle =T^{(1)}_{a \cdot U(\xi)} \langle x,y \rangle $, for all $x,y\in F$.\\
		By page 5 in \cite{Lan95}, we have that $F \langle F, F \rangle$ is dense in $F$ so that by an $\frac{\epsilon}{2}$-argument we get that $U(a\cdot \xi)=a \cdot U(\xi)$, as required.
	\end{proof}

	\begin{Cor} \label{C:equiv-iso}
		Let $E$ and $F$ be $C^*$-correspondences over $\A$ and $\B$, respectively. Suppose that 
		$\A$ (or $\B$) is a subalgebra of compact operators.
Then the following are equivalent
		\begin{enumerate}
			\item
			The $C^*$-correspondences $E$ and $F$ are unitarily isomorphic.
			\item
			$\mathcal{T}_+(E)$ and  $\mathcal{T}_+(F)$ are graded completely isometrically isomorphic.
			\item $\T(E)$ and $\T(F)$ are base-preserving graded *-isomorphic. 
		\end{enumerate}
	\end{Cor}
	
	\begin{proof}
		It is easy to show see that (1) implies (2) (see for instance the proof of \cite[Theorem 4.3 item (1)]{DO18} which works verbatim even when $\A$ and $\B$ are non-commutative), and by Corollary \ref{C:main-square} we have that (2) implies (3). Hence, we need only show that (3) implies (1). When $\A$ is a $C^*$-subalgebra of compact operators this follows from Proposition \ref{p:unitary-iso-cpt}.
	\end{proof}

	
	\section{K-theory}\label{S: K-theory}
	
	In this section we show how stabilized base-preserving s.e.s. isomorphisms induce isomorphisms of six-term short exact sequences of K-groups which only involves the coefficient algebras, Katsura ideals and Cuntz-Pimsner algebras. For the basics of K-theory we refer to \cite{Ror00}. In this section we will rely on K-theory computations in the context of Cuntz-Pimsner algebras from \cite[Section 8]{Kat04b}. Recall that for a $C^*$-correspondence $E$ we denote by $\J_E$ the ideal of relations in $\T(E)$ which coincides with $\K(\F(E)J_E)$.
	
	Let $E$ be a $C^*$-correspondence over $\A$. Denote by $\phi_0 : J_E \rightarrow \K(E^{\otimes 0}J_E) \subseteq \K(\F(E)J_E)$ the left action of the $C^*$-correspondence $E^{\otimes 0}$. By \cite[Proposition 8.1]{Kat04b} we see that $(\phi_0)_* : K_*(J_E) \rightarrow K_*(\J_E)$ is an isomorphism. Now let
	$$
	\begin{tikzcd}
	0 \arrow[r] & \J_E \arrow[r, "j"] & \T(E) \arrow[r] & \O(E) \arrow[r] & 0 
	\end{tikzcd}
	$$
	be the short exact sequence with embedding $j : \J_E \rightarrow \T(E)$. We denote by $i : J_E \rightarrow \A$ the natural embedding. Now let $\iota_{\A} : \A \rightarrow \D_E$ and $\iota_{\K(\E)} : \K(\E) \rightarrow \D_E$ be the $(1,1)$ and $(2,2)$ embedding into the linking algebra, which induce isomorphisms in K-theory. We define $[E] : K_*(J_E) \rightarrow K_*(\A)$ via the composition of $(\phi_E)_* : K_*(J_E) \rightarrow K_*(\K(E))$ and $E_* : K_*(\K(E)) \rightarrow K_*(\A)$, where $\phi_E$ is the left action on $E$ and $E_* = (\iota_{\A})_*^{-1}(\iota_{\K(E)})_*$. Then by the discussion preceding \cite[Theorem 8.6]{Kat04b} we obtain the following commutative diagram 
	$$
	\begin{tikzcd}
	K_*(\J_E) \arrow[r, "j_*"]  
	& K_*(\T(E)) \\
	K_*(J_E) \arrow[u, "(\phi_0)_*"] \arrow[r, "i_* - \text{[}E \text{]}"]
	& K_*(\A) \arrow[swap, u, "(\phi_{\F(E)})_*"]
	\end{tikzcd}
	$$
	Hence, when $F$ is another $C^*$-correspondence over $\B$ such that $\psi : \T(E) \rightarrow \T(F)$ is a base-preserving s.e.s. isomorphism, we denote by $\psi_0$ the restriction of $\psi$ to $\J_E$, by $\rho$ its restriction to $\A$, and by $\overline{\psi} : \O(E) \rightarrow \O(F)$ the induced *-isomorphism on the quotients. Denote $\tau_* : K_*(J_E) \rightarrow K_*(J_F)$ the isomorphism given by $(\phi_0^F)_*^{-1} (\psi_0)_* (\phi_0^E)_*$. We hence obtain the following commutative diagram.
	$$
	\begin{tikzcd}[column sep=scriptsize]
	|[alias = X]| K_*(\J_F) \arrow[to=Z, "(j_F)_*"] & & & |[alias = Z]| K_*(\T(F)) \\
	&  K_*(\J_E) \arrow[r, "(j_E)_*"] \arrow[to=X, "(\psi_0)_*"]
	& K_*(\T(E)) \arrow[to=Z, "\psi_*"] & \\
	&  K_*(J_E) \arrow[u, "(\phi^E_0)_*"] \arrow[r, swap, "i_* - \text{[}E \text{]}"] \arrow[to=W, "\tau_*"]
	& K_*(\A) \arrow[u, "(\phi_{\F(E)})_*"] \arrow[to=Y, "\rho_*"] & \\
	|[alias = W]| K_*(J_F) \arrow[to=X, "(\phi^F_0)_*"] \arrow[to=Y, swap, "i_* - \text{[}F \text{]}"] & & & |[alias = Y]| K_*(\B) \arrow[to=Z, "(\phi_{\F(F)})_*"]
	\end{tikzcd}
	$$
We denote by $\pi_E$ the composition of the embedding $\phi_{\F(E)} : \A \rightarrow \T(E)$ with the quotient map of $\T(E)$ to $\O(E)$. Let $e$ denotes the natural inclusion of any $C^*$-algebra inside its stabilization, by tensoring with a fixed rank-one projection. It is standard that the induced map $e_*$ between associated K-groups is an isomorphism.
	
\begin{Lemma}\label{L: corresp-stab}
Let $E$ be a $C^*$-correspondence over $\A$. Then
\[e_*^{-1}\circ(i^E_*-[E])\circ e_*=i_*^{E\otimes \K}-[E\otimes\K], \]
\end{Lemma}

\begin{proof}
It is readily verified that $e \circ i^E = i^{E \otimes \K} \circ e$, and that $e \circ \phi_E = (\phi_E \otimes \id_{\K}) \circ e = \phi_{E \otimes \K} \circ e$. Thus, from the definitions of $[E]$ and $[E\otimes \K]$, it will suffice to show that $E_* = e_*^{-1} \circ (E \otimes \K)_* \circ e_*$. Similarly to before we have that $e \circ \iota_{\A} = \iota_{\A \otimes \K} \circ e$ and $e \circ \iota_{\K(E)} = \iota_{\K(E) \otimes \K} \circ e$. Thus, we get that
$$
E_* = (\iota_{\A})_*^{-1} \circ (\iota_{\K(E)})_* = 
$$
$$
e_*^{-1} \circ (\iota_{\A \otimes \K})_*^{-1} \circ (\iota_{\K(E \otimes \K)})_* \circ e_* = e_*^{-1} \circ (E \otimes \K)_* \circ e_*,
$$
and the proof is complete.
\end{proof}
	
\begin{Thm} \label{T:K-theory-six-term}
Let $E$ and $F$ be $C^*$-correspondences over $\A$ and $\B$, respectively. Assume $\psi : \T(E \otimes \K) \rightarrow \T(F \otimes \K)$ is a base-preserving s.e.s. isomorphism. Then we have the following commutative diagram:
$$
\begin{tikzcd}
		|[alias = X]| K_1(\O(F)) \arrow[to=Z] & & & |[alias = Z]| K_0(J_F) \arrow[to=U, "i_0 - \text{[}F\text{]}_0"] \\
		& K_1(\O(E)) \arrow[r] \arrow[to=X, "\overline{\psi}_1"] & K_0(J_E) \arrow[d, "i_0 -\text{[}E\text{]}_0"] \arrow[to=Z, "\tau_0"] & \\
		|[alias = V]| K_1(\B) \arrow[to=X, "(\pi_F)_1"] & K_1(\A) \arrow[u, "(\pi_E)_1"] \arrow[l, swap,  "\rho_1"] & K_0(\A) \arrow[d, "(\pi_E)_0"] \arrow[r, "\rho_0"] & |[alias = U]| K_0(\B) \arrow[to=Y, "(\pi_F)_0"] \\
		& K_1(J_E) \arrow[u, "i_1 - \text{[}E\text{]}_1"] \arrow[to=W, "\tau_1"] & K_0(\O(E)) \arrow[l] \arrow[to=Y, "\overline{\psi}_0"] & \\
		|[alias = W]| K_1(J_F) \arrow[to=V, "i_1 - \text{[}F\text{]}_1"] & & & |[alias = Y]| K_0(\O(F)) \arrow[to=W]
		\end{tikzcd}
$$
\end{Thm}

\begin{proof}
By Proposition \ref{P: KatIdealPres} at the level of K-theory we that 
$$
K_*(\T(E)) = K_*(\T(E \otimes K)), \ K_*(\O(E)) = K_*(\O(E\otimes \K))
$$
$$ 
\text{ and } \ K_*(\J_E) = K_*(\J_{E\otimes \K}).
$$
Furthermore, by Corollary \ref{c:tensor-katsura-preserved} and exactness of $\K$ we have that $K_*(J_E) = K_*(J_{E \otimes \K})$. These identifications via $e_*$ are obtained so that 
$$(\phi_{\F(E)})_* = (\phi_{\F(E \otimes \K)})_*, \ (j_E)_* = (j_{E \otimes \K})_*
$$
$$
\text{ and } \ (\phi_0^E)_* = (\phi_0^{E \otimes \K})_*, \ (\pi_E)_* = (\pi_{E \otimes \K})_*.
$$ 
Moreover, by Lemma \ref{L: corresp-stab} we also have $i_*^E - [E] = i_*^{E\otimes \K} - [E \otimes \K]$. Clearly the same also hold for $F$ instead of $E$.

Thus, from the discussion preceding Lemma \ref{L: corresp-stab} combined with \cite[Theorem 8.6]{Kat04b} applied to the $C^*$-corres\-pondences $E \otimes \K$ and $F \otimes \K$, we obtain the desired diagram.
\end{proof}
	
When $E$ and $F$ are $C^*$-correspondences over subalgebras of compact operators $\A$ and $\B$, we are able to compute $\tau_*$. In this case both $J_E$ and $J_F$ are subalgebras of compacts, and must hence be direct sums of algebras of compact operators. From additivity of $K_1$ we get that $K_1(J_E) = K_1(J_F) = \{0\}$ and hence $\tau_1 = 0$. Thus, we need only compute $\tau_0$.

For a $C^*$-algebra $\C$ we use the standard picture of $K_0$ from \cite[Proposition 4.2.2]{Ror00} to express $K_0(\C)$ as differences of equivalence classes $[q]_0 - [s(q)]_0$ for $q \in \P_{\infty}(\C^1)$ where $\C^1$ is the unitization of $\C$ (even if it is unital), and $s : \C^1 \rightarrow \C^1$ is the scalar map (see \cite[Subsection 4.2]{Ror00}). Recall also that if $\varphi : \A \rightarrow \B$ is a *-homomorphism, we denote its unitization by $\varphi^1$.
	
\begin{Prop} \label{P:compute-tau}
Let $E$ and $F$ be $C^*$-correspondences over $\A$ and $\B$, respectively, and assume that $\A$ (or $\B$) are subalgebras of compact operators. If $\psi : \T(E) \rightarrow \T(F)$ is a base-preserving s.e.s. isomorphism, then $\psi(J_E) = J_F$ and $\tau_0 = (\psi|_{J_E})_0$.
	\end{Prop}
	
	\begin{proof}
		Without loss of generality, we assume that $\A=\B$ and $\rho = \psi|_{\A} = \id_{\A}$. Since $\psi_0 : \J_E \rightarrow \J_F$ is a *-isomorphism, by \cite[Corollary 1]{Asa08} there is an $\A$-unitary $U : \F(E)J_E \rightarrow \F(F)J_F$ such that $\psi_0 = \Ad_U$. In particular, 
$$
J_E = \langle \F(E)J_E, \F(E)J_E \rangle =  \langle \F(F)J_F, \F(F)J_F \rangle = J_F
$$
so that the first part is proven. 

Next we show that $\tau_0 = (\id_{J_E})_0 = \id_{K_0(J_E)}$. Let $P_0^F: \F(F)J_E\to J_E$
be the projection onto $F^{\otimes 0}J_E=J_E=E^{\otimes 0}J_E\subseteq \F(F)J_E$. As $\tau_0 = (\phi_0^F)_0^{-1} \circ (\psi_0)_0 \circ (\phi_0^E)_0$, it will suffice to show that $(\Ad_U)_0 \circ (\phi_0^E)_0 = (\phi_0^F)_0$. Here we will abuse notation and simply write $P_0^F$ and $U$ to mean the $n$-direct sums $(P_0^F)^{(n)}$ and $U^{(n)}$ for $n\in \bN$. 

Suppose $q \in \P_{\infty}((J_E)^1)$ is of size $n\times n$. Then
$$
((\Ad_U)_0 \circ (\phi_0^E)_0)\big([q]_0 - [s(q)]_0 \big) = (\Ad_U)_0 \big([P_0^F q P_0^F]_0 - [s(q)P_0^F]_0 \big) = 
$$
$$
[UP_0^F q P_0^F U^*]_0 - [U(s(q)P_0^F)U^*]_0
$$
Now, since $UP_0^F q P_0^F U^* =UP_0^F q P_0^F P_0^F q U^*$ and $UP_0^F q P_0^F$ is in the unitization $\K(\F(F)J_E)^1$, we see that
$$
[UP_0^F q P_0^F U^*]_0 = [P_0^F q P_0^FU^* UP_0^F q P_0^F]_0 = [P_0^F q P_0^F]_0 = ((\phi_0^F)^1)_0([q]_0),
$$
and since $s(q)$ is a projection which commutes with $(P_0^F)^{(n)}$ and $U^{(n)}$ we have that $Us(q)P_0^FU^* = Us(q) P_0^F P_0^F s(q) U^*$ and $s(q)UP_0^F$ is in the unitization $\K(\F(F)J_E)^1$. Hence we also get that
$$
[U(s(q)P_0^F)U^*]_0 = [U s(q)P_0^F P_0^F s(q) U^*]_0 = [s(q)P_0^F]_0 = ((\phi_0^F)^1)_0([s(q)]_0)
$$
Hence, it follows from the standard picture of $K_0$ that $(\Ad_U)_0 \circ (\phi_0^E)_0 = (\phi_0^F)_0$, and we are done.
\end{proof}
	
\section{Hierarchy for graph algebras}\label{S: hierarchy for graph algebras}
	
	We briefly recall the construction of the self-adjoint algebras associated to a directed graph. For more details, we refer the reader to \cite{Rae05} and \cite{DT02}. Let $G=(V,E,r,s)$ be a directed graph with range and source maps $r,s: E \rightarrow V$. We denote by $A_G$ the adjacency matrix for $G$ given by
	$$
	A_G(v,w) = | \{ \ e \in E \ | \ r(e) =v , s(e) =w \ \}|.
	$$
and by $E^{\bullet}$ the collection of all finite paths $\lambda$ in $G$. We also denote by $V^{\reg}$ those vertices $v\in V$ such that $0 < |r^{-1}(v)| < \infty$ and $V^{\fin}$ the vertices $v \in V$ such that $|r^{-1}(v)| < \infty$. We say that $G$ is a \textit{row-finite} graph if $V= V^{\fin}$.
	
	A family $S = (S_v,S_e)_{v\in V,e\in E}$ of operators on Hilbert space $\mathcal{H}$ is a \emph{Toeplitz-Cuntz-Krieger} (TCK) family if
	\begin{enumerate}
		\item
		$\{S_v\}_{v\in V}$ is a set of pairwise orthogonal projections;
		
		\item
		$S_e^*S_e = S_{s(e)}$ for every $e\in E$;
		
		\item
		$\sum_{e\in F} S_eS_e^* \leq S_v$ for every finite subset $F\subseteq r^{-1}(v)$.
	\end{enumerate}
	
	We say that $S$ is a \emph{Cuntz-Krieger} (CK) family if additionally
	\begin{enumerate}
		\item[(4)]
		$\sum_{e \in r^{-1}(v)}S_eS_e^* = S_v$ for every $v\in V^{\reg}$.
	\end{enumerate}
	
	We denote by $\mathcal{T}(G)$ and $\mathcal{O}(G)$ the universal $C^*$-algebras generated by TCK and CK families, respectively. When $G$ is finite with no sinks or sources, $\mathcal{O}(G)$ is the celebrated Cuntz-Krieger algebra of $G$ which is intimately related to the subshift of finite type determined by $G$ (See \cite{CK80}).
	
A natural way to realize $\T(G)$ is by using the left regular TCK family. Let $\H_G := \ell^2(E^{\bullet})$ be the Hilbert space with orthonormal basis $\{ \ \xi_{\lambda} \ | \ \lambda \in E^{\bullet} \ \}$. For each $v\in V$ and $e \in E$ we define
\[
L_v(\xi_{\mu}) = \begin{cases} 
\xi_{\mu} & \text{if } r(\mu) = v \\ 
0 & \text{if } r(\mu) \neq v
\end{cases} \ \ \text{and} \ \
L_e(\xi_{\mu}) = \begin{cases} 
\xi_{e \mu} & \text{if } r(\mu) = s(e) \\ 
0 & \text{if } r(\mu) \neq s(e).
\end{cases}
\]
Then $L=(L_v,L_e)$ is a TCK family and we call it the left regular TCK family. By universality of $\T(G)$ we have a surjective $*$-isomorphism $\pi_L : \T(G) \rightarrow C^*(L)$ which turns out to be injective. Hence $\T(G) \cong C^*(L)$, and we will henceforth identify these algebras without further mention.

For $v\in V$ denote by $\H_{G,v}$ the Hilbert space with orthonormal basis $\{ \ \xi_{\lambda} \ | \ s(\lambda) = v \ \}$. Note that $\H_{G,v}$ is reducing for $L = (L_v,L_e)$, so we denote by $\pi_v : \T(G) \rightarrow B(\H_{G,v})$ the restriction to this subspace. It is easily verified that for any $v\in V$ we have that $L_v - \sum_{e\in r^{-1}(v)}L_e L_e^*$ is the rank one projection onto $\bC \xi_v$. Hence, the ideal generated by $L_v - \sum_{e\in r^{-1}(v)}L_e L_e^*$ for $v \in V^{\fin}$ is $\oplus_{v\in V^{\fin}} \K(\H_{G,v})$, and we denote it by $\I_G$.

\begin{Prop} \label{P: minimal-essential}
Let $G = (V,E)$ be a row-finite directed graph. Then $\I_G$ is a minimum essential ideal in $\T(G)$.
\end{Prop}

\begin{proof}
We first show that $\I_G$ is essential. Let $\I$ be an ideal of $\T(G)$ such that $\I \cap \I_G = \{0\}$. Then we have a natural quotient map $q : \T(G) \rightarrow \T(G) / \I$. However, since $\I$ does not intersect $\I_G = \oplus_{v\in V} \K(\H_{G,v})$, we see that $q(T_v - \sum_{v \in V}T_eT_e^*) \neq 0$ for any $v\in V$. Hence by \cite[Theorem 3.2 \& Corollary 3.3]{DS18} we see that $q \cong \bigoplus_{v\in V} \pi_v^{(\alpha_v)} \oplus \pi_b = \pi_L \oplus \pi_b$ with $\alpha_v \geq 1$ where $\pi_b$ is a representation associated to a CK family. Thus, we get that $q$ is injective, and we must then have that $\I = \{0\}$. 

Next we show that $\I_G$ is a minimum essential ideal. If $\I$ is another essential ideal for $\T(G)$, then $\I \cap \K(\H_{G,v}) \neq \{0 \}$ for all $v\in V$. Hence, we must actually have that $\K(\H_{G,v}) \subseteq \I$. Thus, $\I_G \subseteq \I$ and $\I_G$ is minimum essential.
\end{proof}

Let $G=(V, E)$ be an arbitrary directed graph. From \cite[Chapter 8]{Rae05} we know that $\mathcal{T}(G)$ and $\mathcal{O}(G)$ arise as the Toeplitz-Pimsner and Cuntz-Pimsner algebras of a $C^*$-correspondence over $c_0(V)$. Indeed, define a right pre-Hilbert $c_0(V)$-module structure on finitely supported functions $c_f(E)$ by
	$$(x\cdot a)(e)=x(e)a(s(e)), \text{ for all } x\in c_f(E), a\in c_0(V), e\in E \ , \ \text{ and} $$
	$$ {\langle x,y \rangle} (v)= \sum_{ e\in s^{-1}(v)} \overline{x(e)}y(e), \text{ for all } x,y\in c_f(E), v\in V.$$
	We denote by $X(G)$ the completion of $c_f(E)$ with respect to the induced norm $\|x \|^2 = \| \langle x , x \rangle \|$ defined for $x \in c_f(E)$. Then $X(G)$ becomes a $C^*$-correspondence over $c_0(V)$ by defining the left action $\phi_{X(G)}$ to be
	$$ 
	\phi(a)(x)(e)= a(r(e))x(e), \text{ for all } x\in c_f(E), e\in E, 
	$$ 
	which uniquely extends to a left action on the completion $X(G)$. The $C^*$-correspondence $X(G)$ is called the graph correspondence associated to $G$. In \cite[Chapter 8]{Rae05} it is shown that rigged representations of $X(G)$ are in bijective correspondence with Toeplitz-Cuntz-Krieger families of $G$ and that rigged \emph{covariant} representations of $X(G)$ are in bijective correspondence with Cuntz-Krieger families of $G$. Thus, we see that $\T(G) \cong \T(X(G))$ and that $\O(G) \cong \O(X(G))$, and we treat these realizations interchangeably without mention from now on. The following then generalizes \cite[Theorem 3 (1)]{BLRS19} to arbitrary graphs.

	\begin{Thm}\label{T: Graphs-graded-base-pres}
		Let $G = (V,E)$ and $G' = (V',E')$ be directed graphs. The following are equivalent
		\begin{enumerate}
			\item
			$G$ and $G'$ are isomorphic directed graphs.
			\item
			$X(G)$ and $X(G')$ are unitarily isomorphic $C^*$-correspondences.
			\item
			$\T_+(G)$ and $\T_+(G')$ are graded completely isometrically isomorphic.
			\item
			$\T(G)$ and $\T(G')$ are base-preserving graded *-isomorphic.
		\end{enumerate}		
	\end{Thm}
	
\begin{proof}
Observe that in the case of a graph correspondence $X(G)$, the base algebra $\A=c_0(V)$ is a $C^*$-subalgebra of diagonal compact operators in $B(\ell^2(V))$. Therefore, we can apply Corollary \ref{C:equiv-iso} to conclude that items $(2)-(4)$ are equivalent. Clearly $(1) \implies (2)$ and the converse is proven as follows. If $U : X(G) \rightarrow X(G')$ is a unitary isomorphism of $C^*$-correspondences, implemented by a bijection $\rho : c_0(V) \rightarrow c_0(V')$, then the map $\widehat{\rho} : V \rightarrow V'$ between the spectra is given by $\widehat{\rho}(v) = v'$ where $v' \in V'$ is the unique vertex such that $\rho(p_v) = p_{v'}$. Given $v,w\in V$, note that the subspace $\{ \ x \in X(G) \ | \ x = p_v x p_w \ \}$ has dimension exactly $A_G(v,w)$, and is mapped under $U$ to the subspace $\{ y \in X(G') \ | \ y = p_{v'} y p_{w'} \ \}$ which is of dimension $A_{G'}(v',w')$. Hence, we see that $A_G(v,w) = A_{G'}(v',w')$ so that $G$ and $G'$ are isomorphic graphs via $\widehat{\rho}$.
\end{proof}
	
From \cite[Chapter 8]{Rae05} we know that $J_{X(G)} = c_0(V^{\reg})$ and that $\J_G:= \K(\F(X(G))J_{X(G)}) = \J_{X(G)}$ is the ideal generated by $T_v - \sum_{e\in r^{-1}(v)}T_e T_e^*$ with $v \in V^{\reg}$. Hence we see that $\J_G = \oplus_{v\in V^{\reg}} \K(\H_{G,v})$ under the identification with $\T(G)$ as the $C^*$-algebra generated by $L = (L_v,L_e)$.

\begin{Prop} \label{P: maximum-contained}
Let $G$ be a row-finite graph. Then $\J_G$ is the maximum ideal $\J$ of $\T(G)$ contained in $\I_G$ such that $L_v \notin \J$ for any $v\in V$.
\end{Prop}

\begin{proof}
We already know that $L_v \notin \J_G$ for all $v\in V$. Next, if $\J$ is an ideal contained in $\I_G$ such that $L_v \notin \J$ for all $v\in V$, as $\J \subseteq \I_G = \oplus_{v\in V}\K(\H_{G,v})$ we must have that $\J = \oplus_{v \in V'} \K(\H_{G,v})$ for some subset $V' \subseteq V$. However when $v\in V$ is a source we have that $\K(\H_{G,v})$ is the ideal generated by $L_v$. Since $L_v \notin \J$ we must have that $\K(\H_{G,v}) \cap \J = \{0\}$ so that $V' \subseteq V^{\reg}$ and hence $\J \subseteq \J_G$.
\end{proof}
	
In what follows, we refer to \cite{DT02} for additional details. It is clear $K_0(\A) \cong \bigoplus_{v\in V} \bZ$ and $K_0(J_{X(G)}) \cong \bigoplus_{v\in V^{\reg}} \bZ$. Let $A_G = \begin{bmatrix} B_G & C_G \\ * & * \end{bmatrix}$ denote the $2\times 2$ block decomposition of $A_G$ according to $V^{\reg}$ and $V^{\sing} := V \setminus V^{\reg}$.
	By \cite[Theorem 3.1]{DT02}, the map $i_0 -[X(G)] : K_0(J_{X(G)}) \rightarrow K_0(\A)$ is identified with the matrix $\begin{bmatrix} B_G^t - I\\ C_G^t \end{bmatrix}$. Note that when $G$ is \emph{row-finite}, we get that $A_G = \begin{bmatrix} B_G & C_G \\ 0 & 0 \end{bmatrix}$, so that $G$ is completely determined by the map $i_0 -[X(G)]$.
	
We next prove a substantial strengthening of Theorem \ref{T: Graphs-graded-base-pres} in the row-finite case. This showcases the strength of applying the hierarchy established in Corollary \ref{C:main-square} to resolve stable isomorphism problems for non-self-adjoint algebras via techniques from K-theory.
	
\begin{Thm} \label{T: stable-graphs-base-pres}
Let $G=(V,E)$ and $G' = (V',E')$ be row-finite directed graphs. The following are equivalent.
	\begin{enumerate}
			\item
			$G$ and $G'$ are isomorphic directed graphs.
			\item
			$\T_+(G)$ and $\T_+(G')$ are completely isometrically isomorphic.
			\item
			$\T_+(G) \otimes \K$ and $\T_+(G') \otimes \K$ are completely isometrically isomorphic.
			\item
			$\T(G)$ and $\T(G')$ are base-preserving *-isomorphic.
			\item
			$\T(G) \otimes \K$ and $\T(G') \otimes \K$ are base-preserving *-isomorphic.
	\end{enumerate}
\end{Thm}

\begin{proof}
It is clear that $(1)$ implies any one of $(2) - (5)$, that $(2)$ implies $(3)$ and that $(4)$ implies $(5)$. By Corollary \ref{C:main-square} we also have that $(2)$ implies $(4)$ and that $(3)$ implies $(5)$ with the addition of Proposition \ref{P: KatIdealPres}. Hence, it will suffice to show that $(5)$ implies $(1)$.

Suppose $\varphi : \T(G)\otimes \K \rightarrow \T(G')\otimes \K$ is a base-preserving *-isomorphism. By Proposition \ref{P: KatIdealPres} we may identify $\varphi$ with an isomorphism of the Toeplitz algebras $\T(X(G)\otimes \K)$ and $\T(X(G')\otimes \K)$. Hence, without loss of generality we may assume that $V= V'$, that $c_0(V)\otimes \K=c_0(V')\otimes \K$ as subalgebras of $\T(X(G)\otimes \K)$ and $\T(X(G')\otimes \K)$ and that $\varphi|_{c_0(V)\otimes \K}=\id_{c_0(V)\otimes\K}$.

By Propositions \ref{P: minimal-essential} and \ref{P: maximum-contained} we see that $\J_G$ (and $\J_{G'}$) is a maximum ideal $\J$ contained in a minimum essential ideal such that $L_v \notin \J$ for all $v\in V$. Since $\varphi|_{c_0(V)\otimes \K}=\id_{c_0(V)\otimes\K}$ and tensoring with $\K$ preserves the lattice of ideals, we see that $\varphi$ must map $\J_G \otimes \K$ to $\J_{G'} \otimes \K$. Thus, it follows that $\varphi$ is a s.e.s. base preserving *-isomorphism.
		
By Theorem \ref{T:K-theory-six-term} and the discussion preceding our theorem we have that $i_0^{X(G)}-[X(G)]$ equals $i_0^{X(G')}-[X(G')]$. 
Since these two maps determine $A_G$ and $A_{G'}$ respectively, it follows that $G$ and $G'$ must be isomorphic directed graphs.
\end{proof}

The next example shows that the implications $(5) \implies (3)$ and $(4) \implies (2)$ in Theorem \ref{T: stable-graphs-base-pres} can fail for graphs that are not row-finite.

\begin{Exl} \label{ex:non-iso-base-iso}
Consider the two graphs $G$ and $G'$ given by adjacency matrices
\[
\begin{bmatrix}0&\infty\\\infty&0\end{bmatrix}\text{ and } \begin{bmatrix}\infty&\infty\\\infty&0\end{bmatrix}
\]
respectively. We visualize these graphs as
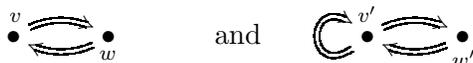
\[
\xymatrix{\bullet\ar@/^/@{=>}[r]^(0.1){v}&\bullet\ar@/^/@{=>}[l]^(0.1){w}&}\text{ and }\xymatrix{&\bullet\ar@(dl,ul)@{=>}[]\ar@/^/@{=>}[r]^(0.1){v'}&\bullet\ar@/^/@{=>}[l]^(0.1){w'}}
\]
where double edges have infinite multiplicity. We remark that these graphs differ by the so-called ``move (T)'' as described in \cite{ERS12}. Since all vertices are singular, their associated graph $C^*$-algebras coincide with their associated Toeplitz-Cuntz-Krieger algebras, and are isomorphic due to \cite[Lemma 3.6]{ERS12}. Tracing through the $*$-isomorphisms that are concretely given in \cite[Lemma 3.1 and Lemma 3.6]{ERS12}, we see that they send $S_v$ to $S_{v'}$ and 
$S_w$ to $S_{w'}$, so that  $ \T(G)$ and $ \T(G')$ are base-preserving $^*$-isomorphic.

This shows the necessity of requiring the graphs be row-finite in Theorem \ref{T: stable-graphs-base-pres}. Furthermore, note that the isomorphism is trivially s.e.s. since $\O(G)=\T(G)$ and $\O(G')=\T(G')$. Since these are Toeplitz-Pimsner and Cuntz-Pimsner algebras, this example also shows that condition $(4)$ and $(5)$ in Corollary \ref{C:main-square} are not equivalent. Lastly, by \cite[Theorem 2.11]{KK04} we see that conditions $(5)$ and $(3)$ in Corollary \ref{C:main-square} are also not equivalent.
\end{Exl}
	
\subsection*{Acknowledgments} The third author would like to express her gratitude to her supervisors Prof. Ilan Hirshberg and Prof. Wilhelm Winter for their helpful advice and generous support. She is especially thankful for the various operator algebras seminars, courses and discussions held in M\"{u}nster during her time spent there.
	

\end{document}